\documentclass[preprint,11pt,authoryear]{elsarticle}
\usepackage{amsfonts}
\usepackage{graphicx,amssymb,amsmath,amsthm,url}
\usepackage{color}

\newtheorem{thm}{Theorem}
\newtheorem{prop}{Proposition}
\newtheorem{lem}{Lemma}
\newtheorem{defi}{Definition}
\newtheorem{rem}{Remark}
\newtheorem{ex}{Example}

\journal{}
\input{tcilatex}
\begin{document}

\begin{frontmatter}
\title{Dynamics of complex-valued fractional-order neural networks}

\author[IeAT,UVT]{Eva Kaslik\corref{cor}}
\ead{ekaslik@gmail.com}
\author[IeAT,UPT]{Ileana Rodica R\u{a}dulescu}
\ead{nicola\_rodica@yahoo.com}

\address[IeAT]{Institute e-Austria Timisoara, Bd. V. Parvan nr. 4, room 045B, 300223, Timisoara, Romania}

\address[UVT]{Department of Mathematics and Computer Science, West University of Timisoara, Romania}

\address[UPT]{Faculty of Applied Sciences, University Politehnica of Bucharest, Romania}

\cortext[cor]{Corresponding author}

\begin{abstract}
The dynamics of complex-valued fractional-order neuronal networks are investigated, focusing on stability, instability and Hopf bifurcations. Sufficient conditions for the asymptotic stability and instability of a steady state of the network are derived, based on the complex system parameters and the fractional order of the system, considering simplified neuronal connectivity structures (hub and ring). In some specific cases, it is possible to identify the critical values of the fractional order for which Hopf bifurcations may occur. Numerical simulations are presented to illustrate the theoretical findings and to investigate the stability of the limit cycles which appear due to Hopf bifurcations.
\end{abstract}


\end{frontmatter}

\section{Introduction}

In the last few decades, generalizations of dynamical systems using
fractional derivatives instead of classical integer-order derivatives have
proved to be more accurate in the mathematical modeling of real world
phenomena arising from several interdisciplinary areas. It is now
well-understood that fractional derivatives provide a good tool for the
description of memory and hereditary properties of various processes,
fractional-order systems being characterized by infinite memory.
Phenomenological description of viscoelastic liquids \citep{Heymans_Bauwens}%
, colored noise \citep{Cottone}, diffusion and wave propagation %
\citep{Henry_Wearne,Metzler}, boundary layer effects in ducts %
\citep{Sugimoto}, fractional kinetics \citep{Mainardi_1996}, electromagnetic
waves \citep{Engheia}, electrode-electrolyte polarization %
\citep{Ichise_Nagayanagi_Kojima}, represent just a few of the many
application areas of fractional derivatives.

It is important to emphasize that many qualitative properties of
integer-order dynamical systems cannot be extended by simple generalizations
to fractional-order dynamical systems, and hence, the analysis of
fractional-order dynamical systems is a very important field of research.
For example, it has been shown \citep{Eva-Siva-non-periodic} that the
fractional-order derivative (of Caputo, Grunwald-Letnikov or
Riemann-Liouville type) of a non-constant periodic function cannot be a
periodic function of the same period, while the integer-order derivative of
a periodic function is indeed a periodic function of the same period. As a
consequence, exact periodic solutions do not exist in a wide class of
fractional-order dynamical systems. In fact, many fractional-order analogues
of important theoretical results from the classical integer-order dynamical
systems theory are still open questions (such as the Hopf bifurcation
theorem for fractional-order systems or the stability analysis of
time-delayed fractional-order systems \citep{Kaslik_2014_ICNPAA}).
Therefore, the theoretical analysis of fractional-order models arising from
different real world problems has to be done with care.

In fractional-order neural network models, the common capacitor from the
continuous-time integer-order recurrent neural networks is replaced by a
generalized capacitor, called fractance \citep{Nakagawa,Fractance}. In many
engineering applications, there is a need for lossy capacitors with
prescribed losses, in order to accomplish analog fractional calculus
operations within a single device. A review of circuit theory approaches
aimed at creating fractional-order capacitors (fractance devices) has been
presented by \citet{Elwakil}.

The fractional-order formulation of artificial neural network models is also
justified by research results concerning biological neurons. For example, %
\citet{Lundstrom} concluded that fractional differentiation provides neurons
with a fundamental and general computation ability that can contribute to
efficient information processing, stimulus anticipation and
frequency-independent phase shifts of oscillatory neuronal firing. The
results reported by \citet{Anastasio} suggest that the resulting net output
of motor and premotor neurons can be described as fractional differentiation
relative to eye position.

Fractional-order real-valued artificial neural networks have been in the
spotlight since the year 2000, starting with the pioneering works of %
\citet{Arena,Matsuzaki,Petras,Boroomand}, which mainly report on results of
numerical simulations, especially on the numerical evidence of limit cycles
and chaotic phenomena. Several early papers also discuss chaotic
synchronization in fractional-order neural networks %
\citep{Zhu_2008,Zhou_2008_1,Zhou_2009}. The first papers devoted to the
theoretical stability analysis and Hopf bifurcations of fractional-order
neural networks of Hopfield type \citep{ES-IJCNN-2011,Eva-Siva-4} also
describe potential routes towards the onset of chaotic behavior when the
fractional order of the system increases. The numerical examples presented
in these papers unveiled highly complex dynamical behavior in real-valued
fractional-order neural networks, such as the co-existence of strange
attractors with several asymptotically stable steady states and limit
cycles. Moreover, in the last five years, a large number of papers have been
published in this field, focusing on theoretical topics such as global
Mittag-Leffler stability and synchronization \citep{Chen_2014}, undamped
oscillations generated by Hopf bifurcations \citep{Xiao_2015}, dynamics of
delayed fractional-order neural networks \citep{Chen_2013,Wang_2014}, etc.

{However, since many interesting applications of neural networks involve
complex signals, such as pattern recognition and classification, intelligent
image processing, nonlinear filtering, brain-computer interfaces, time
series prediction, robotics and bioinformatics, the investigation of
complex-valued neural networks (CVNNs) is an essential extension of the
analysis of real-valued neural networks. Complex-valued neural networks (CVNNs) are networks that use
complex-valued variables and parameters, successfully dealing in this manner with
complex-valued information.

CVNNs are exceptionally rich in diversity and
because they are very compatible with wave phenomena, they are appropriate
for the processes associated with complex altitude (e.g. interferometric
radar systems). Usually, propagation and interference of electromagnetic
waves are expressed by the magnitude of transmission and reflection, phase
progression and retardation, superposition of fields and so on, phenomena
which might be naturally expressed by the use of complex numbers \citep{Hirose_2006}.
Correspondingly, these phenomena are correlated with basic processes in the
CVNNs, for instance weighting at synaptic connections, i.e. multiplications
in amplitude and shifts in phase, and summation of the weighted inputs.
CVNNs provide systems with appropriate information representations in many
other fields, most of them being related to wave phenomena, for example:
measurements using waves such as radar image processing, active antennas in
electromagnetism, analysis and synthesis in voice processing, learning
electronic-wave devices etc.

Several important research directions have arisen concerning CVNNs, such as: the formal generalization
of commonly used algorithms to the complex-valued case, the use of original
complex-valued activation functions that can increase significantly the
neuron and network functionality and the development of quaternion neurons
and neural networks \citep{Hirose_2009}. }

To the best of our knowledge, there are only a few research papers dedicated
to the investigation of fractional-order complex-valued neural networks,
published very recently. Sufficient conditions for finite-time stability %
\citep{Rakkiyappan_2014_finite}, uniform asymptotic stability %
\citep{Rakkiyappan_2015_fcvnn}, $O(t^{-\alpha})$-stability and global
asymptotical periodicity \citep{Rakkiyappan_2016} have been obtain for
fractional-order complex-valued neural networks with time-delays. In %
\citet{Rakkiyappan_2015_stability}, linearization techniques have been used
to obtain sufficient conditions for the asymptotic stability of the
equilibrium states of fractional-order complex-valued neural networks with
time-delays. Necessary conditions for the synchronization of
fractional-order complex-valued neural networks with time delays has been
obtained by \citet{Bao_2016}. Results concerning bifurcation phenomena in delayed fractional complex-valued neural networks have been recently reported by \citet{Huang_2017}.

It is important to emphasize that, to the best of our knowledge, at this
moment, there are no known results concerning Hopf bifurcations in
fractional-order complex-valued neural networks, and therefore, this
constitutes one of the aims of this paper. Moreover, this paper is devoted
to the theoretical stability analysis of fractional-order complex-valued
neural networks of Hopfield type, extending the results presented by %
\citet{ES-IJCNN-2011,Eva-Siva-4} for fractional-order real-valued neural
networks. Two special connectivity types will be discussed in detail:
networks with hub and ring structures, respectively. These simplified
connectivity structures are studied to gain insight into the mechanisms
underlying the behavior of recurrent networks with more complicated
connectivity. Studying patterns of interconnections, called "network motifs" %
\citep{Milo}, occurring in neural networks is fundamental to understanding
the dynamic behavior of the whole network.

The paper is structured as follows: in section 2, a few preliminaries are
included about fractional-order differential systems, while in section 3, a
basic results are described regarding complex-valued fractional-order neural
networks of Hopfield type. In section 4, a detailed stability and Hopf
bifurcation analysis is undertaken for fractional-order Hopfield neural
networks, concentrating on the two previously mentioned connectivity
structures (hub and ring). In each case, numerical examples are also
presented. Concluding remarks are included in section 5.

\section{Preliminaries on fractional-order differential systems}

The fractional derivate employed in this paper is the Caputo derivative,
which is widely considered more applicable to real world problems, as it
only requires initial conditions given in terms of integer-order
derivatives, representing well-understood features of physical situations %
\citep{Podlubny}.

\begin{defi}
For a continuous function $f$, with $f'\in L^1_{loc}(\mathbb{R}^+)$, the
Caputo fractional-order derivative of order $q\in(0,1)$ of $f$ is defined by
\begin{equation*}
^cD^q f(t)=\frac{1}{\Gamma(1-q)}\int_ 0^t(t-s)^{-q}f^{\prime }(s)ds~,
\end{equation*}
where the gamma function is defined, as usual, as:
\begin{equation*}
\Gamma(z)=\int_0^\infty e^{-t}t^{z-1}dt~.
\end{equation*}
\end{defi}

Remarkable scientific books which provide the main theoretical tools for the
qualitative analysis of fractional-order dynamical systems, and at the same
time, show the interconnection as well as the contrast between classical
differential equations and fractional differential equations, are the works
of \citet{Podlubny,Kilbas,Lak}.

The following stability result holds for linear autonomous fractional-order
systems (see \citet{Matignon,Sabatier_2012}):

\begin{thm}
\label{thm.linear.stab} The linear autonomous system
\begin{equation}\label{eq.lin}
^cD^q \mathbf{x}=A \mathbf{x}
\end{equation}
where $A\in\mathbb{R}^{n\times n}$ and $q\in(0,1)$ is asymptotically stable
if and only if
\begin{equation}  \label{eq.lambda.spec}
|\arg(\lambda)|>\frac{q\pi}{2}\qquad\forall \lambda\in\sigma(A)
\end{equation}
or equivalently, if and only if
\begin{equation}  \label{ineq.lambda}
|\Im(\lambda)|>\Re(\lambda)\tan\frac{q\pi}{2}\qquad\forall
\lambda\in\sigma(A)
\end{equation}
where $\sigma(A)$ denotes the spectrum of the matrix $A$ (i.e. the set of
all eigenvalues), $\Re(\lambda)$ and $\Im(\lambda)$ denote respectively, the
real and imaginary part of $\lambda$.
\end{thm}

\begin{rem}\label{rem.real.comp}
{The integer order linear system $\dot{%
\mathbf{x}}=A \mathbf{x}$ is asymptotically stable if and only if $%
\Re(\lambda)<0$ for any $\lambda\in\sigma(A)$. Moreover, it is easy to see
that inequality (\ref{ineq.lambda}) is satisfied for any $%
\lambda\in\sigma(A) $ with $\Re(\lambda)<0$ and $q\in(0,1)$.

Therefore, if the integer order linear system $\dot{\mathbf{x}}=A \mathbf{x}$ is asymptotically stable, it follows that the fractional order system (\ref{eq.lin}) is also asymptotically stable, for any $q\in(0,1)$.

The converse of the above statement is generally not true. However, in the special case when all the eigenvalues of the matrix $A$ are real (for example if $A$ is a symmetric matrix), the inequality (\ref{ineq.lambda}) is satisfied if and only if all the eigenvalues of $A$ are strictly negative. We conclude that in this case, asymptotic stability of the fractional-order linear system (\ref{eq.lin}) is equivalent to the asymptotic stability of the integer order system $\dot{\mathbf{x}}=A \mathbf{x}$.

In general, for $0<q_1<q_2\leq 1$, if the linear system $^cD^{q_2} \mathbf{x}=A \mathbf{x}$ is asymptotically stable, from Theorem \ref{thm.linear.stab} it follows that
$$|\arg(\lambda)|>\frac{q_2\pi}{2}>\frac{q_1\pi}{2},\quad\forall~\lambda\in \sigma(A)$$
and hence, the linear system $^cD^{q_1} \mathbf{x}=A \mathbf{x}$ is also asymptotically stable.}
\end{rem}

\section{Complex-valued fractional-order HNNs}

A real-valued fractional-order neural network model of Hopfield type (FHNN)
with Caputo-type derivatives, introduced by \citet{ES-IJCNN-2011,Eva-Siva-4}%
, is represented by the following system
\begin{equation}
^{c}D^{q}x_{k}(t)=-a_{k}x_{k}(t)+\sum_{j=1}^{n}T_{kj}g_{j}(x_{j}(t))+I_{k},%
\quad \forall k=\overline{1,n},~\forall ~t>0  \label{frac.Hop}
\end{equation}%
where $q\in (0,1)$, $a_{k}>0$ are the self-regulating parameters of the
neurons, ${T}=({T}_{kj})_{n\times n}\in \mathbb{R}^{n\times n}$ is the
interconnection matrix, $g_{k}:\mathbb{R}\rightarrow \mathbb{R}$ are the
neuron input-output activation functions and ${I}_{k}\in \mathbb{R}$ denote
the external inputs.

In this paper, we generalize the previously considered model (\ref{frac.Hop}%
), introducing the following complex-valued neural network model with Caputo
fractional-order derivatives, described by the following system:
\begin{equation}
^{c}D^{q}z_{k}(t)=-a_{k}z_{k}(t)+\sum_{j=1}^{n}T_{kj}g_{j}(z_{j}(t))+I_{k},%
\quad \forall k=\overline{1,n},~\forall ~t>0  \label{frac.Hop.comp}
\end{equation}%
where $z_{k}:\mathbb{R}^{+}\rightarrow\mathbb{C}$ are the complex state
variables, $a_{k}>0$ are the self-regulating parameters of the neurons, ${T}%
=({T}_{kj})_{n\times n}\in\mathbb{C}^{n\times n}$, is the complex
interconnection matrix, $g_{k}:\mathbb{C} \rightarrow \mathbb{C}$ are
complex-valued activation functions and ${I}_{k}\in\mathbb{C}$ represent the
complex external inputs.

Several types of activation functions which are often used in complex-valued
neural networks are described by \citet{Kuroe_2003}. In particular, %
\citet{Georgiou_1992} describe the properties of the following complex
activation function
\begin{equation}  \label{activ.fun}
g(z)=\frac{z}{c_1+c_2|z|}\qquad\text{with }c_1,c_2>0.
\end{equation}
which proves to be useful in many practical applications.

Denoting $\mathbf{z}(t)=(z_{1}(t),z_{2}(t),...,z_{n}(t))^{T}$, $A=\text{diag}%
(a_{1},a_{2},...,a_{n})\in \mathbb{R}^{n\times n}$, $\mathbf{g}(\mathbf{z}%
)=(g_{1}(z_{1}),g_{2}(z_{2}),...,g_{n}(z_{n}))^{T}$ and $\mathbf{I}%
=(I_{1},I_{2},...,I_{n})^{T}\in \mathbb{C}^{n}$, the system (\ref%
{frac.Hop.comp}) can be written in the following vector form:
\begin{equation}
^{c}D^{q}\mathbf{z}(t)=-A\mathbf{z}(t)+T\mathbf{g}(\mathbf{z}(t))+\mathbf{I}.
\label{frac.Hop.comp.vec}
\end{equation}

Considering the real and the imaginary parts of the complex state vector,
interconnection matrix, input vector and activation functions respectively,
we denote
\begin{align*}
\mathbf{z}(t)& =\mathbf{x}(t)+i \mathbf{y}(t),\text{ where }\mathbf{x},\mathbf{y}:\mathbb{R}^{+}\rightarrow \mathbb{R}^{n},
\\
T& =T^{R}+iT^{I},\text{ where }T^{R},T^{I}\in \mathbb{R}^{n\times n}, \\
\mathbf{I}& =\mathbf{I}^{R}+i\mathbf{I}^{I},\text{ where }\mathbf{I}^{R},%
\mathbf{I}^{I}\in \mathbb{R}^{n}, \\
\mathbf{g}(\mathbf{z})& =\mathbf{g}(\mathbf{x}+i\mathbf{y})=\mathbf{g}^{R}(%
\mathbf{x},\mathbf{y})+i\mathbf{g}^{I}(\mathbf{x},\mathbf{y}) \\
& \text{with }\left\{
\begin{array}{c}
\mathbf{g}^{R}(\mathbf{x},\mathbf{y})=\left(
g_{1}^{R}(x_{1},y_{1}),g_{2}^{R}(x_{2},y_{2}),...,g_{n}^{R}(x_{n},y_{n})%
\right) ^{T} \\
\mathbf{g}^{I}(\mathbf{x},\mathbf{y})=\left(
g_{1}^{I}(x_{1},y_{1}),g_{2}^{I}(x_{2},y_{2}),...,g_{n}^{I}(x_{n},y_{n})%
\right) ^{T}%
\end{array}%
\right.
\end{align*}

With the above notations, system (\ref{frac.Hop.comp.vec}) becomes
\begin{equation*}
^{c}D^{q}(\mathbf{x}+i\mathbf{y}) =-A(\mathbf{x}+i\mathbf{y}%
)+\left(T^{R}+iT^{I}\right)\left(\mathbf{g}^{R}(\mathbf{x},\mathbf{y})+i%
\mathbf{g}^{I}(\mathbf{x},\mathbf{y})\right) +\mathbf{I}^{R}+i\mathbf{I}^{I}.
\end{equation*}
which is equivalent to the following real-valued fractional-order system:
\begin{equation}
\left\{
\begin{array}{c}
^{c}D^{q}\mathbf{x}(t)=-A\mathbf{x}+T^{R}\mathbf{g}^{R}(\mathbf{x},\mathbf{y}%
)-T^{I}\mathbf{g}^{I}(\mathbf{x},\mathbf{y})+\mathbf{I}^{R} \\
^{c}D^{q}\mathbf{y}(t)=-A\mathbf{y}+T^{R}\mathbf{g}^{I}(\mathbf{x},\mathbf{y}%
)+T^{I}\mathbf{g}^{R}(\mathbf{x},\mathbf{y})+\mathbf{I}^{I}.%
\end{array}
\right.  \label{frac.Hop.comp.r}
\end{equation}%
Denoting
\begin{align*}
\mathbf{u}(t) &=(\mathbf{x}(t),\mathbf{y}(t))^{T}\in\mathbb{R}^{2n}, \\
\tilde{A} &=\left(
\begin{array}{cc}
A & 0 \\
0 & A%
\end{array}
\right) = \text{diag}\left(
a_{1},a_{2},...,a_{n},a_{1},a_{2},...,a_{n}\right) \in\mathbb{R}^{2n\times
2n}, \\
\tilde{T} &=\left(
\begin{array}{cc}
T^{R} & -T^{I} \\
T^{I} & T^{R}%
\end{array}%
\right) \in \mathbb{R}^{2n\times 2n}, \\
\tilde{\mathbf{g}}(\mathbf{u})&=\tilde{\mathbf{g}}((\mathbf{x},\mathbf{y}%
)^T)=\left(\mathbf{g}^{R}(\mathbf{x},\mathbf{y}),\mathbf{g}^{I}(\mathbf{x},%
\mathbf{y})\right) ^{T}= \\
&=\left(
g_{1}^{R}(x_{1},y_{1}),..,g_{n}^{R}(x_{n},y_{n}),g_{1}^{I}(x_{1},y_{1}),..,g_{n}^{I}(x_{n},y_{n})\right)^{T}\in%
\mathbb{R}^{2n} \\
\tilde{\mathbf{I}}&=(\mathbf{I}^R,\mathbf{I}^I)^T\in \mathbb{R}^{2n},
\end{align*}
the $n$-dimensional complex system (\ref{frac.Hop.comp.vec}) is then
equivalent to the $2n$-dimensional fractional-order real-valued system
\begin{equation}
^{c}D^{q}\mathbf{u}(t)=-\tilde{A}\mathbf{u}(t)+\tilde{T}\mathbf{\tilde{g}}(%
\mathbf{u}(t))+\mathbf{\tilde{I}}.  \label{frac.Hop.comp.vect.2n}
\end{equation}
It has to be emphasized that system (\ref{frac.Hop.comp.vect.2n}) is
equivalent to a real-valued bidirectional associative memory (BAM) network
if and only if
\begin{equation*}
g_j^R(x_j,y_j)=g_j^R(y_j)\quad\text{and}\quad g_j^I(x_j,y_j)=g_j^I(x_j),\quad%
\text{for any }j=\overline{1,n},
\end{equation*}
i.e., the real part of the complex activation function $g_j$ depends only on
the imaginary part $y_j$ of the state variable, and the imaginary part of
the of $g_j$ depends only on the real part $x_j$ of the state variable.

\section{Stability and bifurcations}

In the following, let us consider $\mathbf{z}^{\ast }=\mathbf{x}^{\ast }+i%
\mathbf{y}^{\ast }\in\mathbb{C} ^{n}$ an equilibrium state of the
complex-valued fractional-order neural network (\ref{frac.Hop.comp.vec}):
\begin{equation*}
-A\mathbf{z}^{\ast }+T\mathbf{g}(\mathbf{z}^{\ast })+\mathbf{I}=0.
\end{equation*}
Equivalently, $\mathbf{u}^{\ast }=(\mathbf{x}^{\ast },\mathbf{y}^{\ast
})^{T} $ is a steady state of the real system (\ref{frac.Hop.comp.vect.2n}),
i.e. a solution of
\begin{equation*}
-\tilde{A}\mathbf{u}^{\ast }+\tilde{T}\tilde{\mathbf{g}}(\mathbf{u}^{\ast })+%
\tilde{\mathbf{I}}=0.
\end{equation*}
Obviously, $\mathbf{z}^{\ast }\in \mathbb{C}^{n}$ is a steady state of the
system (\ref{frac.Hop.comp.vec}) with the fractional order $q\in (0,1)$ if
and only if it is an equilibrium state of the corresponding integer order
system (i.e. for $q=1$). Therefore, the same results hold for the existence,
uniqueness or multiplicity of equilibrium states of fractional-order neural
networks, as in the case of integer-order neural networks. However, the
conditions for the asymptotic stability of an equilibrium state $\mathbf{z}%
^{\ast }\in \mathbb{C}^{n}$ are in general more strict in the case of the
corresponding integer-order system.

With the aim of studying the stability of the equilibrium state $\mathbf{z}%
^\ast$ in the framework of the fractional-order system (\ref%
{frac.Hop.comp.vec}), or equivalently, the stability of the equilibrium
state $\mathbf{u}^\ast$ of system (\ref{frac.Hop.comp.vect.2n}), we rely on
the linearization theorem recently proved by \citet{Li_Ma_2013}. This
linearization theorem is an analogue of the classical Hartman-Grobman
theorem for the case of integer-order dynamical systems. For a rigorous
application of this linearization theorem, we have to require that the
function $\tilde{\mathbf{g}}$ is of class $C^1$ (continuously
differentiable) on a neighborhood of the steady state $\mathbf{u}^\ast\in%
\mathbb{R}^{2n}$. Using the notations $g_{k}:\mathbb{C}\rightarrow\mathbb{C}$%
, with $g_k(z)=g_{k}(x+iy)=g_{k}^{R}(x,y)+ig_{k}^{I}(x,y)$, for the complex
activation functions, this is equivalent to the following assumption:
\begin{equation*}
\mathbf{(A_1)}\quad g_k^R \text{ and } g_k^I \text{ are of class $C^1$ in a
neighborhood of } (x_k^\ast,y_k^\ast)\in\mathbb{R}^2,~\forall~k=\overline{1,n}.
\end{equation*}

Due to the fact that the activation functions $g_k$ are usually assumed to
be bounded, we cannot require them to be holomorphic on the whole complex
plane $\mathbb{C}$ (i.e. entire functions). The reason is that, according to
Liouville's theorem, every bounded entire function must be constant.
Therefore, for simplicity, we only assume the following:
\begin{equation*}
\mathbf{(A_2)}\quad g_k \text{ are complex-differentiable at the point }
z_k^\ast=x_k^\ast+iy_k^\ast,~\forall~k=\overline{1,n}.
\end{equation*}
Hence, the Cauchy-Riemann conditions are satisfied:
\begin{equation}
\left\{
\begin{array}{l}
\dfrac{\partial g_{k}^{R}}{\partial x}(x_k^\ast,y_k^\ast)=\dfrac{\partial
g_{k}^{I}}{\partial y}(x_k^\ast,y_k^\ast), \\
\dfrac{\partial g_{k}^{R}}{\partial y}(x_k^\ast,y_k^\ast)=-\dfrac{\partial
g_{k}^{I}}{\partial x}(x_k^\ast,y_k^\ast).%
\end{array}%
\right.  \label{CR.cond}
\end{equation}

\begin{rem}
The activation function $g(z)$ given by (\ref{activ.fun}) satisfies assumption $(A_1)$. In fact, the real and imaginary parts $g^R$ and $g^I$ are of class $C^1$ on $\mathbb{R}^2$. A detailed analysis has been performed by \citet{Georgiou_1992}. Moreover, the function $g(z)$ is  complex-differentiable at $0$ and the complex derivative is $g'(0)=1$.
\end{rem}

Based on the previous assumptions and the Hartman-Grobman-type linearization
theorem \citep{Li_Ma_2013}, the asymptotic stability of the steady state $%
\mathbf{u}^{\ast }\in \mathbb{R}^{2n}$ of the fractional-order system (\ref%
{frac.Hop.comp.vect.2n}) is determined by the eigenvalues of the Jacobian
matrix
\begin{equation}
\tilde{J}(\mathbf{u}^{\ast })=-\tilde{A}+\tilde{T}D\tilde{\mathbf{g}}(%
\mathbf{u}^{\ast }),  \label{jacob1}
\end{equation}%
where
\begin{equation*}
D\tilde{\mathbf{g}}(\mathbf{u}^{\ast })=\left(
\begin{array}{cc}
D\mathbf{g}_{x}^{R}(\mathbf{u}^{\ast }) & D\mathbf{g}_{y}^{R}(\mathbf{u}%
^{\ast }) \\
-D\mathbf{g}_{y}^{R}(\mathbf{u}^{\ast }) & D\mathbf{g}_{x}^{R}(\mathbf{u}%
^{\ast })%
\end{array}%
\right)
\end{equation*}%
is a block matrix obtained using conditions (\ref{CR.cond}):
\begin{eqnarray*}
D\mathbf{g}_{x}^{R}(\mathbf{u}^{\ast }) &=&\text{diag}\left( \frac{\partial
g_{1}^{R}}{\partial x_{1}}(x_{1}^{\ast },y_{1}^{\ast }),\frac{\partial
g_{2}^{R}}{\partial x_{2}}(x_{2}^{\ast },y_{2}^{\ast }),...,\frac{\partial
g_{n}^{R}}{\partial x_{n}}(x_{n}^{\ast },y_{n}^{\ast })\right) , \\
D\mathbf{g}_{y}^{R}(\mathbf{u}^{\ast }) &=&\text{diag}\left( \frac{\partial
g_{1}^{R}}{\partial y_{1}}(x_{1}^{\ast },y_{1}^{\ast }),\frac{\partial
g_{2}^{R}}{\partial y_{2}}(x_{2}^{\ast },y_{2}^{\ast }),...,\frac{\partial
g_{n}^{R}}{\partial y_{n}}(x_{n}^{\ast },y_{n}^{\ast })\right) .
\end{eqnarray*}%
Furthermore, the Jacobian matrix becomes
\begin{equation*}
\tilde{J}(\mathbf{u}^{\ast })=-\tilde{A}+\tilde{T}D\mathbf{\tilde{g}}(\mathbf{u}^{\ast })=
-\left(
\begin{array}{cc}
A & 0 \\
0 & A%
\end{array}%
\right) +\left(
\begin{array}{cc}
T^{R} & -T^{I} \\
T^{I} & T^{R}%
\end{array}%
\right) \left(
\begin{array}{cc}
D_{1} & -D_{2} \\
D_{2} & D_{1}%
\end{array}%
\right) ,
\end{equation*}
where $A=\text{diag}\left(a_{1},..,a_{n}\right)$ , $D_{1}=D\mathbf{g}_{x}^{R}(\mathbf{u}^{\ast
})$, $D_{2}=-D\mathbf{g}_{y}^{R}(\mathbf{u}^{\ast })$ are real diagonal matrices.

In order to simplify the computations, we use the following notations
\begin{eqnarray*}
U &=&-A+T^{R}D_{1}-T^{I}D_{2}, \\
V &=&T^{R}D_{2}+T^{I}D_{1}.
\end{eqnarray*}%
Consequently, the Jacobian matrix (\ref{jacob1}) becomes%
\begin{equation*}
\tilde{J}(\mathbf{u}^{\ast })=\left(
\begin{array}{cc}
U & -V \\
V & U%
\end{array}%
\right) .
\end{equation*}

The characteristic polynomial of the Jacobian matrix $\tilde{J}(\mathbf{u}%
^{\ast })$ is%
\begin{align*}
P(\lambda ) &=\det (\tilde{J}(\mathbf{u}^{\ast })-\lambda I_{2n}) &  \\
&=\det \left(
\begin{array}{cc}
U-\lambda I_{n} & -V \\
V & U-\lambda I_{n}%
\end{array}%
\right) & (\text{Row}_1+i\cdot\text{Row}_2) \\
&=\det \left(
\begin{array}{cc}
U-\lambda I_{n}+iV & -V+iU-i\lambda I_{n} \\
V & U-\lambda I_{n}%
\end{array}%
\right) & (\text{Col}_2-i\cdot\text{Col}_1) \\
&=\det \left(
\begin{array}{cc}
U-\lambda I_{n}+iV & O_{n} \\
V & U-\lambda I_{n}-iV%
\end{array}%
\right) &  \\
&=\det (U+iV-\lambda I_{n})\cdot\det (U-iV-\lambda I_{n}). &
\end{align*}
Consequently, the roots of the characteristic polynomial $P(\lambda )$ are
either the eigenvalues of the matrix $U+iV$, or the eigenvalues of the
matrix $U-iV=\overline{U+iV}.$ In the following, let us denote $M=U+iV$ and $%
\sigma (M)$ the spectrum of $M$. Then, $\sigma (\tilde{J}(u^{\ast }))=\sigma (M)\cup
\sigma (\overline{M})=\sigma (M)\cup \overline{\sigma (M)}$, so we only need
to determine the eigenvalues of the matrix $M$.

We can easily see that
\begin{align}
M& =-A+T^{R}D_{1}-T^{I}D_{2}+i\left( T^{R}D_{2}+T^{I}D_{1}\right) =
\label{matrixM} \\
& =-A+\left( T^{R}+iT^{I}\right) \left( D_{1}+iD_{2}\right) =  \notag \\
& =-A+T\mathbf{g}^{\prime }(\mathbf{z}^{\ast }),  \notag
\end{align}%
where $\mathbf{g}^{\prime }(\mathbf{z}^{\ast })$ is the complex Jacobian
matrix of the function $\mathbf{g}$ at the equilibrium state $\mathbf{z}%
^{\ast }$:
\begin{equation*}
\mathbf{g}^{\prime }(\mathbf{z}^{\ast })=\text{diag}\left( g_{1}^{\prime
}(z_{1}^{\ast }),g_{2}^{\prime }(z_{2}^{\ast }),...,g_{n}^{\prime
}(z_{n}^{\ast })\right) .
\end{equation*}

Based on the above reasoning and Theorem \ref{thm.linear.stab}, we conclude:

\begin{prop}
\label{prop.complex.stab} If assumptions $(A_1)$ and $(A_2)$ hold, the
equilibrium state $\mathbf{z}^\ast$ of system (\ref{frac.Hop.comp.vec}) is
asymptotically stable if and only if all the eigenvalues of the
complex-valued matrix $-A+T\mathbf{g}^{\prime }(\mathbf{z}^{\ast })$ satisfy
$|\arg(\lambda)|>\frac{q\pi}{2}$.
\end{prop}



In the following section, the stability and bifurcation properties will be
investigated in the case of a special network structure, called the hub
structure.

\subsection{Complex-valued FHNN with hub structure}

In scale-free networks \citep{Barabasi_1999}, some nodes, called "hubs",
have many more connections than other nodes. In fact, the network as a whole
has a power-law distribution of the number of links connecting to a node (at
least asymptotically). In this type of networks, the existence of hub
structures is a common feature, playing a fundamental role in defining the
connectivity of the scale-free networks and in characterizing their
dynamical behavior. The dynamics of hubs for integer-order real-valued
neural networks was studied by \citet{Kitajima-Kurths}. Real-valued
fractional-order neural networks with hub structure have been investigated
by \citet{Eva-Siva-4} in the non-delayed case and by \citet{Wang_2014_ring}
in the delayed case.

In the following, we consider the complex-valued fractional-order neural
network of $n\geq 3$ neurons with hub structure
\begin{equation}
\left\{ \!\!\!%
\begin{array}{l}
^{c}D^{q}z_{1}(t)=-az_{1}(t)+\sum_{j=1}^{n}T_{1j}g_{j}(z_{j}(t))+I_{1} \\
^{c}D^{q}z_{j}(t)=-b_{j}z_{j}(t)+T_{j1}g_{1}(z_{1}(t))+T_{jj}g_{j}(z_{j}(t))+I_{j},\qquad\forall~j=\overline{2,n}
\end{array}%
\right.  \label{eq.hub}
\end{equation}%
where $a,b_{j}>0$, $T_{jk}\in \mathbb{\mathbb{C}}$, $\forall j,k=\overline{1,n}$. Here, the first neuron (called the central neuron) is the center of the
hub, and all the other $(n-1)$ neurons (called peripheral neurons) are
connected directly only to the central neuron and to themselves
(self-connections are present). The interconnection matrix of this neural
network is
\begin{equation*}
T=\left(
\begin{array}{ccccc}
T_{11} & T_{12} & T_{13} & ... & T_{1n} \\
T_{21} & T_{22} & 0 & ... & 0 \\
T_{31} & 0 & T_{33} & ... & 0 \\
... & ... & ... & ... & ... \\
T_{n1} & 0 & 0 & ... & T_{nn}%
\end{array}%
\right) \in \mathbb{C}^{n\times n},
\end{equation*}

Let $\mathbf{z}^{\ast }=(z_{1}^{\ast },z_{2}^{\ast },...,z_{n}^{\ast })^{T}$
an equilibrium of the system (\ref{eq.hub}) and, for simplicity, we denote

\begin{equation}
-a+T_{11}g_{1}^{\prime }(z_{1}^{\ast })=\alpha ,  \label{eq.alpha}
\end{equation}%
and assume that%
\begin{equation}
-b_{j}+T_{jj}g_{j}^{\prime }(z_{j}^{\ast })=\beta ,\quad \forall j=\overline{%
2,n}.  \label{eq.beta}
\end{equation}%
For example, if the peripheral neurons are identical, assumption (\ref%
{eq.beta}) holds always true.

Consequently, from (\ref{matrixM}), (\ref{eq.alpha}) and (\ref{eq.beta}) the
Jacobian matrix of the system (\ref{eq.hub}) is
\begin{equation*}
J(\mathbf{z}^{\ast })=\left(
\begin{array}{ccccc}
\alpha & T_{12}g_{2}^{\prime }(z_{2}^{\ast }) & T_{13}g_{3}^{\prime
}(z_{3}^{\ast }) & ... & T_{1n}g_{n}^{\prime }(z_{n}^{\ast }) \\
T_{21}g_{1}^{\prime }(z_{1}^{\ast }) & \beta & 0 & ... & 0 \\
T_{31}g_{1}^{\prime }(z_{1}^{\ast }) & 0 & \beta & ... & 0 \\
... & ... & ... & ... & ... \\
T_{n1}g_{1}^{\prime }(z_{1}^{\ast }) & 0 & 0 & ... & \beta%
\end{array}%
\right) .
\end{equation*}%
Using the notation%
\begin{equation}
\gamma =\sum_{k=2}^{n}T_{1k}T_{k1}g_{1}^{\prime }(z_{1}^{\ast
})g_{k}^{\prime }(z_{k}^{\ast }),  \label{eq.gamma}
\end{equation}
we obtain the characteristic equation
\begin{equation}
(\lambda -\beta )^{n-2}\left( \lambda ^{2}-(\alpha +\beta )\lambda +\alpha
\beta -\gamma \right) =0,  \label{eq.char.hub}
\end{equation}%
with $\alpha $, $\beta $, $\gamma \in \mathbb{C}$ defined above depending on
the equilibrium $\mathbf{z}^{\ast }$ and the system parameters $a$, $b_{i}$,
$T_{ij}$ of (\ref{eq.hub}).

\begin{rem}\label{rem.spectrum.hub}
Let $\lambda_1,\lambda_2$ denote the roots of the second order polynomial with complex coefficients
$$P(\lambda)= \lambda ^{2}-(\alpha +\beta )\lambda +\alpha\beta -\gamma.$$
If $n\geq 3$, it is easy to see that $\lambda_0=\beta$ is a root of the characteristic equation (\ref{eq.char.hub}), with the order of multiplicity $(n-2)$. Therefore:
$$\sigma(J(\mathbf{z}^\ast))=\{\beta,\lambda_1,\lambda_2\}.$$
On the other hand, in the particular case $n=2$, we simply obtain
$$\sigma(J(\mathbf{z}^\ast))=\{\lambda_1,\lambda_2\}.$$
\end{rem}

In the following, with the aim of analyzing the stability of the steady state $\mathbf{z}^{\ast }$
using the characteristic equation (\ref{eq.char.hub}) and Theorem \ref%
{thm.linear.stab}, we will concentrate our attention on the roots of the second-order polynomial $P(\lambda)$ defined in Remark \ref{rem.spectrum.hub}. We denote:
\begin{align*}
\rho_{1}&=\left\vert \alpha +\beta \right\vert & \theta _{1}&=Arg(\alpha
+\beta )\in(-\pi,\pi] \\
\rho_{2}&=\left\vert \alpha \beta -\gamma \right\vert>0 &
\theta_{2}&=Arg(\alpha \beta -\gamma )\in(-\pi,\pi]
\end{align*}
and
\begin{align*}
A&=\theta _{1}-\dfrac{\theta _{2}}{2}\in\left(-\frac{3\pi}{2},\frac{3\pi}{2}\right)\\
B&=\dfrac{1}{2}\arccos
\left\{ \dfrac{1}{2}\left[ \dfrac{\rho _{1}^{2}}{2\rho _{2}}-\sqrt{\dfrac{%
\rho _{1}^{4}}{4\rho _{2}^{2}}+4-2\dfrac{\rho _{1}^{2}}{\rho _{2}}\cos (2A)}%
\right] \right\}\in\left[0,\frac{\pi}{2}\right].
\end{align*}

The following result can be proved using standard mathematical tools:

\begin{lem}
\label{lem1.stab.hub} Let $\phi_{1},\phi_{2}\in(-\pi,\pi]$ denote
the principal arguments of the roots of the polynomial $P(\lambda)$.
The following cases occur:
\begin{itemize}
\item[(a)] If $\cos(A)\geq 0$, then $\displaystyle \phi_{1,2}=\dfrac{\theta _{2}}{2}\pm B$;

\item[(b)] If $\cos(A)<0$, then
\begin{itemize}
\item[(b.1)] if $\theta_2\leq -2B$ then $\displaystyle \phi_{1,2}=\dfrac{\theta _{2}}{2}+\pi\pm B$;
\item[(b.2)] if $\theta_2\in(-2B,2B]$ then $\displaystyle \phi_{1,2}=\dfrac{\theta _{2}}{2}\pm (\pi-B)$;
\item[(b.3)] if $\theta_2>2B$ then $\displaystyle \phi_{1,2}=\dfrac{\theta _{2}}{2}-\pi\pm B$;
\end{itemize}
\end{itemize}
\end{lem}
\begin{proof}
See Appendix.
\end{proof}

Based on Remark \ref{rem.spectrum.hub} and Lemma \ref{lem1.stab.hub}, we obtain the following result regarding
the asymptotic stability of the equilibrium point $\mathbf{z}^{\ast }$.

\begin{prop}
\label{prop.stab.hub}
Let $q^{\ast\ast}$ be defined as follows:
$$
q^{\ast\ast}=\left\{
               \begin{array}{ll}
                 \displaystyle\frac{1}{\pi}\min\{|\theta_2-2B|,|\theta_2+2B|\}, & \textrm{if }\cos(A)\geq 0 \\\\
                 \displaystyle\frac{1}{\pi}\min\{|\theta_2+2\pi-2B|,|\theta_2+2\pi+2B|\}, & \textrm{if }\cos(A)< 0\textrm{ and }\theta_2\leq -2B \\\\
                 \displaystyle\frac{1}{\pi}\min\{|\theta_2+2\pi-2B|,|\theta_2-2\pi+2B|\}, & \textrm{if }\cos(A)< 0\textrm{ and }\theta_2\in(- 2B,2B] \\\\
                 \displaystyle\frac{1}{\pi}\min\{|\theta_2-2\pi+2B|,|\theta_2-2\pi-2B|\}, & \textrm{if }\cos(A)< 0\textrm{ and }\theta_2> 2B
               \end{array}
             \right.
$$
The steady state $\mathbf{z}^{\ast }$ of the system (\ref{eq.hub}) is
asymptotically stable if and only if $q\in (0,q^{\ast })\cap (0,1)$, where
\begin{equation}\label{qstar.hub}
q^\ast=\left\{
           \begin{array}{ll}
             \displaystyle\min\left\{\frac{2}{\pi}|Arg(\beta)|,q^{\ast\ast}\right\} &,~ \textrm{if }n\geq 3  \\
             q^{\ast\ast} & ,~\textrm{if }n=2
           \end{array}
         \right.
\end{equation}
\end{prop}

\begin{proof}
The steady state $\mathbf{z}^{\ast }$ of the system (\ref{eq.hub}) is asymptotically stable if and only if the arguments of the all the roots of the characteristic equation (\ref{eq.char.hub}) satisfy inequality (\ref{eq.lambda.spec}) from Theorem \ref{thm.linear.stab}. If $n\geq 3$, with the notations from Remark \ref{rem.spectrum.hub}, we know that these roots are $\{\beta,\lambda_1,\lambda_2\}$. On the other hand, when $n=2$, the characteristic roots are only $\{\lambda_1,\lambda_2\}$. Based on Lemma \ref{lem1.stab.hub}, we obtain the desired result.
\end{proof}

\begin{rem}
The case $n=2$ corresponds to the case of simple hub containing a central neuron and only one
peripheral neuron described by a simple two-dimensional complex-valued neural
network%
\begin{equation}
\left\{ \!\!\!%
\begin{array}{l}
^{c}D^{q}z_{1}(t)=-az_{1}(t)+T_{11}g_{1}(z_{1}(t))+T_{12}g_{2}(z_{2}(t))+I_{1}
\\
^{c}D^{q}z_{2}(t)=-bz_{2}(t)+T_{21}g_{1}(z_{1}(t))+T_{22}g_{2}(z_{2}(t))+I_{2}%
\end{array}%
\right.   \label{eq.2d}
\end{equation}
The stability analysis of the steady state $\mathbf{z}^{\ast }$ requires the
study of the characteristic equation
\begin{equation}
P(\lambda)=\lambda ^{2}-(\alpha +\beta )\lambda +\alpha \beta -\gamma =0.
\label{eq.char.hub.2d}
\end{equation}%
where $\alpha =-a+T_{11}g_{1}^{\prime }(z_{1}^{\ast })$, $\beta
=-b+T_{22}g_{2}^{\prime }(z_{2}^{\ast })$ and $\gamma
=T_{12}T_{21}g_{1}^{\prime }(z_{1}^{\ast })g_{2}^{\prime }(z_{2}^{\ast })$.
Based on Proposition \ref{prop.stab.hub}, we obtain that the steady state $\mathbf{z}^{\ast }$ of (\ref{eq.2d}) is asymptotically stable if and only if
any $q\in (0,q^{\ast\ast})\cap(0,1)$.
\end{rem}

It is worth mentioning that in the case of fractional-order dynamical
systems very few theoretical results are known at this time regarding
bifurcation phenomena. \citet{ElSaka} attempt to formulate conditions for
the occurrence of Hopf bifurcations, based on observations arising from
numerical simulations. However, the complete characterization of the Hopf
bifurcation in fractional-order dynamical systems, as well as the stability
of the resulting limit cycle, are still open questions.

\begin{rem}
Because a steady state $\mathbf{z}^{\ast }$ of (\ref{eq.hub}) does not
depend on the fractional order $q\in (0,1)$, a good choice for the
bifurcation parameter that may be considered in system (\ref{eq.hub}) is the
fractional order $q$ itself.

According to Proposition \ref{prop.stab.hub}, if $q^\ast<1$, the Jacobian
matrix $J(\mathbf{z}^\ast)$ has a pair of complex eigenvalues $\pm \frac{%
q^\ast\pi}{2}$, and according to \citet{ElSaka}, this corresponds to a Hopf
bifurcation in the fractional-order neural network (\ref{eq.hub}). The
steady state $\mathbf{z}^\ast$ of (\ref{eq.hub}) is asymptotically stable if
and only if $q\in(0,q^\ast)$. As $q$ increases and crosses the critical
value $q^\ast$, the steady state $\mathbf{z}^\ast$ becomes unstable, and a
limit cycle is expected to appear in a neighborhood of $\mathbf{z}^\ast$,
due to the Hopf bifurcation phenomenon.

However, if $q^\ast\geq 1$, it follows that $\mathbf{z}^\ast$ is
asymptotically stable for any $q\in(0,1)$.
\end{rem}

\begin{ex}
The following complex-valued fractional-order neural network of three
neurons with hub structure is considered:
\begin{equation}  \label{ex.hub}
\left\{%
\begin{array}{l}
^cD^{q} z_1(t)=-z_1(t)+(2-5i)g(z_1(t))-(2+i)g(z_2(t))+(2+i)g(z_3(t)) \\
^cD^{q} z_2(t)=-2z_2(t)+3g(z_1(t))+(1+i)g(z_2(t)) \\
^cD^{q} z_3(t)=-2z_3(t)+(1-i)g(x_1(t))+(1+i)g(z_3(t))%
\end{array}
\right.
\end{equation}
where the activation function is $g(z)=\displaystyle\frac{z}{1+|z|}$, with
the complex derivative $g^{\prime }(0)=1$. In this neural network, $a=1$, $%
b_2=b_3=2$, $T_{11}=2-5i$, $T_{22}=T_{33}=1+i$, $T_{12}=-2-i$$, T_{21}=3$, $%
T_{13}=2+i$ and $T_{31}=1-i$.

For the steady state $\mathbf{z}^\ast=\mathbf{0}$, the parameters given by
the equations (\ref{eq.alpha}), (\ref{eq.beta}) and (\ref{eq.gamma}) are $%
\alpha=1-5i$, $\beta=-1+i$ and $\gamma=-3-4i$. We can easily compute the
critical value of the fractional order $q$ given by (\ref{qstar.hub}), $%
q^\ast=0.844976$. Hence, the null solution is asymptotically stable if and
only if $q\in(0,0.844976)$ and unstable for $q\in(0.844976,1)$. At $q=q^\ast$%
, a Hopf bifurcation is expected to take place.

Indeed, when the fractional order increases above the critical value $%
q^\ast=0.844976$, numerical simulations show the appearance of an
asymptotically stable limit cycle in a neighborhood of the origin (see Fig %
\ref{fig:hub}). For all numerical simulations, the generalization of the
Adams-Bashforth-Moulton predictor-corrector method has been used, described
by \citet{Diethelm}. Numerical simulations also show the existence of an
asymptotically stable limit cycle for any $q\in(q^\ast,1)$.

\begin{figure}[tbp]
\centering
\begin{tabular}{ccc}
\includegraphics[width=0.47\textwidth]{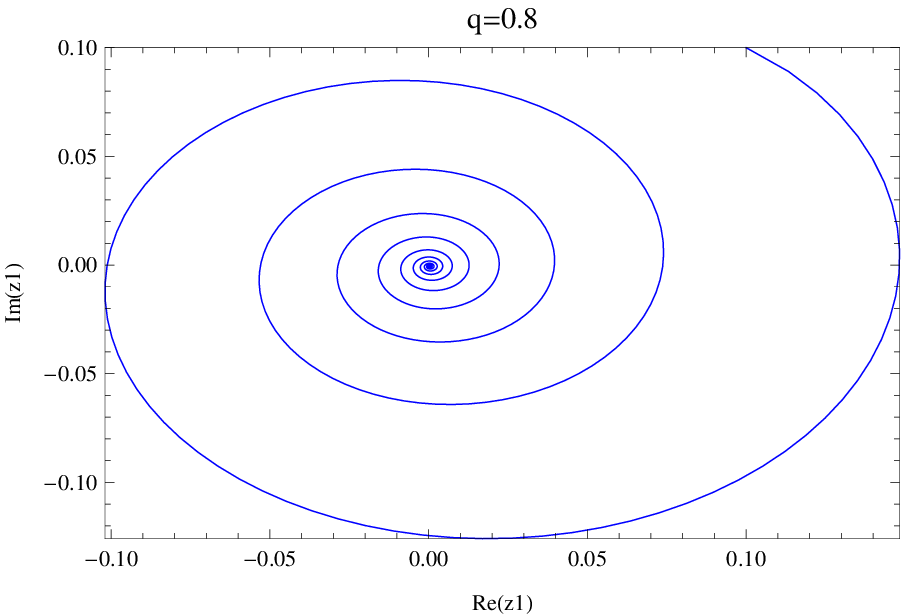} & %
\includegraphics[width=0.47\textwidth]{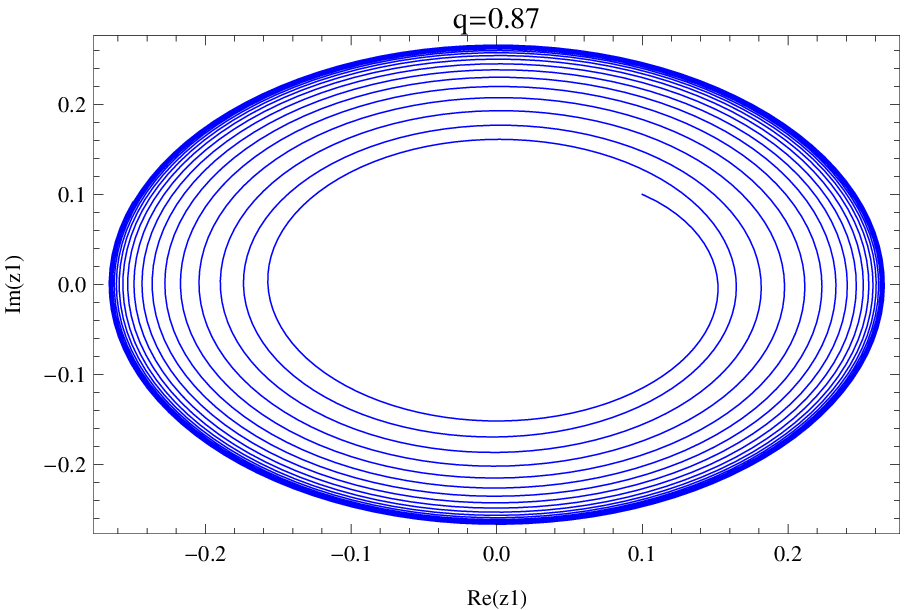} &  \\
&  &
\end{tabular}
.
\caption{Hopf bifurcation occurring in a neighborhood of the null solution
of system (\protect\ref{ex.hub}) at the critical value of the fractional
order $q^\ast=0.844976$. The null solution is asymptotically stable for $%
q=0.8$ (left) and unstable for $q=0.87$ (right). An asymptotically stable
limit cycle appears due to the supercritical Hopf bifurcation as $q$
increases above the critical value (right).}
\label{fig:hub}
\end{figure}
\end{ex}

\subsection{Complex-valued FHNN with ring structure}

Another simple connectivity structure frequently considered in neural
networks is the ring structure, in which every neuron of the network is only
connected to its two closest neighbors. Ring architectures have been found
in a variety of neural structures, such as the hippocampus, cerebellum,
neocortex, and even in chemistry and electrical engineering. It is a
well-known fact that the real cortical connectivity pattern in these network
structures is extremely sparse: most connections are between nearby cells,
and long-range connections become progressively more infrequent %
\citep{Hirsch}.

The dynamics of continuous-time neural networks with ring structure,
including aspects such as stability, bifurcations, patters of oscillations,
nonlinear waves, have been extensively studied by %
\citet{ring-Baldi-Atiya,ring-Campbell-Ruan-Wei,ring-Guo-Huang-4,ring-Guo,ring-Wei-Jiang,ring-Guo-Huang-3,ring-Guo-Huang-2,ring-Guo-Huang-1,ring-Bungay-Campbell,ring-Lu-Guo}%
. Bifurcation phenomena and chaotic behavior in discrete-time delayed neural
rings have been discussed by \citet{KB6,Eva_IJCNN2009}.

Real-valued fractional-order neural networks with ring structure, with and
without delays, have been recently analyzed by %
\citet{Eva-Siva-4,Wang_2014_ring,Wang_2015_ring}.

Considering, for simplicity, a null external input, we analyze a
complex-valued fractional-order neural network with ring structure described
by:
\begin{equation}  \label{frac.ring}
\left\{%
\begin{array}{lr}
^cD^qz_1(t)= & -a_1z_1(t)+T_{1,1}g_1(z_1(t))+\hfill \\
& +T_{1,2}g_2(z_2(t))+T_{1,n}g_n(z_n(t))\hfill \\
^cD^qz_j(t)= & -a_2z_j(t)+T_{j,j-1}g_{j-1}(z_{j-1}(t))+\hfill \\
& +T_{j,j}g_j(z_j(t))+T_{j,j+1}g_{j+1}(z_{j+1}(t))\qquad,~ \forall~j=%
\overline{2,n-1} \\
^cD^qz_n(t)= & -a_nz_n(t)+T_{n,1}g_1(z_1(t))+\hfill \\
& +T_{n,n-1}g_{n-1}(z_{n-1}(t))+T_{n,n}g_n(z_n(t))\hfill%
\end{array}
\right.
\end{equation}
where $n\geq 3$. The interconnection matrix of system (\ref{frac.ring}) is
\begin{equation}  \label{eq.B.circ}
T=\left(
\begin{array}{cccccccc}
T_{1,1} & T_{1,2} & 0 & 0 & ... & 0 & 0 & T_{1,n} \\
T_{2,1} & T_{2,2} & T_{2,3} & 0 & ... & 0 & 0 & 0 \\
0 & T_{3,2} & T_{3,3} & T_{3,4} & ... & 0 & 0 & 0 \\
... & ... & ... & ... & ... & ... & ... & ... \\
0 & 0 & 0 & 0 & ... & T_{n-2,n-2} & T_{n-2,n-1} & 0 \\
0 & 0 & 0 & 0 & ... & T_{n-1,n-2} & T_{n-1,n-1} & T_{n-1,n} \\
T_{n,1} & 0 & 0 & 0 & ... & 0 & T_{n,n-1} & T_{n,n} \\
&  &  &  &  &  &  &
\end{array}
\right)\in\mathbb{C}^{n\times n}
\end{equation}

Let $\mathbf{z}^\ast$ be a steady state of system (\ref{frac.ring}). We will
make the following simplifying assumptions:
\begin{equation}  \label{eq.alpha.ring}
-a_j+T_{j,j}g_j^{\prime }(z_j^\ast)=\alpha,\quad\forall~j=\overline{1,n}
\end{equation}
\begin{equation}  \label{eq.beta.ring}
T_{j,j+1}g_{j+1}^{\prime }(z_{j+1}^\ast)=\beta,\quad\forall~j=\overline{1,n}
\end{equation}
and
\begin{equation}  \label{eq.gamma.ring}
T_{j,j-1}g_{j-1}^{\prime }(z_{j-1}^\ast)=\gamma,\quad\forall~j=\overline{1,n}
\end{equation}
where the complex parameters $\alpha$, $\beta$ and $\gamma$ depend on the
steady state $\mathbf{z}^\ast$ which is being analyzed and the system
parameters $a_j$, $T_{j,k}$ of (\ref{frac.ring}).

It is easy to see that these assumptions are satisfied in the case of
identical neurons, having the same activation function $g_j=g_0$, $j=%
\overline{1,n}$, the same self-regulating parameters $a_j=a_0>0$, $j=%
\overline{1,n}$, and the connection weights satisfying $T_{j,j}=T_0$
(self-connection), $T_{j,j+1}=T_1$ (forward connection) and $T_{j,j-1}=T_2$
(backward connection), for any $j=\overline{1,n}$. With these assumptions,
we may consider steady states $\mathbf{z}^\ast$ with equal complex scalar
components $z_1^\ast=z_2^\ast=...=z_n^\ast=z^\ast$, such that
\begin{equation}  \label{eq.ring.scalar}
-a_0z^\ast+(T_0+T_1+T_2)g_0(z^\ast)=0.
\end{equation}
Obviously, $z^\ast=0$ is a solution of this equation, as the trivial
solution is a steady state of (\ref{frac.ring}). Besides the null solution,
the equation (\ref{eq.ring.scalar}) may have other solutions as well, and
therefore, the system (\ref{frac.ring}) may have multiple steady states.

Let us return to the stability analysis of a steady state $\mathbf{z}^\ast$
of system (\ref{frac.ring}). If the assumptions (\ref{eq.alpha.ring}), (\ref%
{eq.beta.ring}) and (\ref{eq.gamma.ring}) are fulfilled, we obtain that the
Jacobian matrix of system (\ref{frac.ring}) is a circulant matrix of the form
\begin{equation*}
J(\mathbf{z}^\ast)=\text{circ}(\alpha,\beta,0,...,0,\gamma)
\end{equation*}
The eigenvalues of the complex circulant matrix $J(\mathbf{z}^\ast)$ are (see for example \citet{Circulant} and \citet{Kra2012circulant}):
\begin{equation*}
\lambda_p=\alpha+\beta\omega^p+\gamma\overline{\omega}^p,\quad
\forall~p\in\{0,1,..,n-1\}
\end{equation*}
where $\omega=\exp{\left(\frac{2\pi i}{n}\right)}$.

We easily obtain the following sufficient condition for the asymptotic
stability of the steady state $\mathbf{z}^\ast$:

\begin{prop}
\label{prop.stab.ring.simple} If conditions (\ref{eq.alpha.ring}), (\ref%
{eq.beta.ring}) and (\ref{eq.gamma.ring}) are satisfied and
\begin{equation}  \label{cond.as.stab.ring.gen}
\Re(\alpha)+|\beta+\overline{\gamma}|<0,
\end{equation}
the steady state $\mathbf{z}^\ast$ of system (\ref{frac.ring}) is
asymptotically stable, regardless of the fractional order $q\in(0,1)$.
\end{prop}

\begin{proof} We have:
\begin{align*}
\Re(\lambda_p)&=\Re(\alpha)+\Re(\beta\omega^p+\gamma\overline{\omega}^p)=\\
&=\Re(\alpha)+\Re((\beta+\overline{\gamma})\omega^p)\leq\\
&=\Re(\alpha)+|(\beta+\overline{\gamma})\omega^p|=\\
&=\Re(\alpha)+|\beta+\overline{\gamma}|<0,
\end{align*}
for any $p\in\{0,1,..,n-1\}$, and hence, all eigenvalues of the Jacobian matrix $J(\mathbf{z}^\ast)$ are in the left half-plane. We obtain that $\bold{z}^\ast$ is asymptotically stable, regardless of the fractional order $q\in(0,1)$.
\end{proof}

Moreover, based on Proposition \ref{prop.complex.stab}, the following result
holds:

\begin{prop}
\label{prop.stab.ring.q} If conditions (\ref{eq.alpha.ring}), (\ref%
{eq.beta.ring}) and (\ref{eq.gamma.ring}) are satisfied and
\begin{equation}  \label{cond.as.stab.ring.q}
\Re(\alpha)+|\beta+\overline{\gamma}|\geq 0,
\end{equation}
the steady state $\mathbf{z}^\ast$ of system (\ref{frac.ring}) is
asymptotically stable if and only if
\begin{equation}  \label{ineq.q.stab.ring}
q<\frac{2}{\pi}\min\limits_{p=\overline{0,n-1}}\left|\arg(\alpha+\beta%
\omega^p+\gamma\overline{\omega}^p)\right|,
\end{equation}
where $\omega=\exp{\left(\frac{2\pi i}{n}\right)}$. If the critical value
\begin{equation*}
q^\star=\frac{2}{\pi}\min\limits_{p=\overline{0,n-1}}\left|\arg(\alpha+\beta%
\omega^p+\gamma\overline{\omega}^p)\right|\in(0,1)
\end{equation*}
a Hopf bifurcation is expected to occur in system (\ref{frac.ring}) in a
neighborhood of the steady state $\mathbf{z}^\ast$ .
\end{prop}

Taking into consideration the fact that the parameters $\alpha$, $\beta$ and
$\gamma$ given by (\ref{eq.alpha.ring}), (\ref{eq.beta.ring}) and (\ref%
{eq.gamma.ring}) belong to the complex plane, due to the high complexity of
the problem, it is a very difficult task to determine a simple expression
for the value of the minimum involved in the inequality (\ref%
{ineq.q.stab.ring}) or the formula of the critical value $q^\star$. In the particular case when $\alpha$, $\beta$ and
$\gamma$ take real values, we refer to \citet{Eva-Siva-4} for a more detailed analysis.

In the following, we exemplify a particular complex-valued case which will be
investigated in detail.

\begin{ex}
For simplicity, let us consider $\mathbf{z}^\ast=\mathbf{0}$. Moreover, we assume that the self-regulating parameters of the neurons are equal to $1$, no
neuronal self-connection is present, and $\beta$ and $\gamma$ belong to the
unit circle, i.e.:
\begin{equation*}
\alpha=-1,\quad \beta=\exp(i\theta_1),\quad \gamma=\exp(i\theta_2),\quad
\text{where }\theta_1,~\theta_2\in(-\pi,\pi].
\end{equation*}
Let $S_q(\mathbf{z}^\ast)\subset(-\pi,\pi]\times (-\pi,\pi]$ denote the set of parameters $(\theta_1,\theta_2)$ for which the equilibrium state $\mathbf{z}^\ast=\mathbf{0}$ is an asymptotically stable equilibrium state of system (\ref{frac.ring}). The set $S_q(\mathbf{z}^\ast)$ is called stability domain of the equilibrium state $\mathbf{z}^\ast$ of system (\ref{frac.ring}), with respect to the fractional order $q$.

It can be easily seen that the eigenvalues $\lambda_p$, $p=\overline{0,n-1}$, of the Jacobian matrix $J(\mathbf{z}^\ast)$ can be expressed as:
\begin{equation*}
\lambda_p=-1+2\cos\left(\frac{\theta_1-\theta_2}{2}+\frac{2p\pi}{n}%
\right)\exp\left(i\frac{\theta_1+\theta_2}{2}\right)
\end{equation*}
Based on Proposition \ref{prop.stab.ring.simple}, we deduce that if $%
\displaystyle\cos(\theta_1+\theta_2)<-\frac{1}{2}$ (or equivalently $%
\displaystyle \frac{2\pi}{3}<|\theta_1+\theta_2|<\frac{4\pi}{3}$) then $%
\Re(\lambda_p)<0$ and therefore, the steady state $\mathbf{z}^\ast$ is
asymptotically stable, for any $q\in(0,1]$. Therefore, we deduce:
\begin{equation*}
D_{\theta_1,\theta_2}=\{(\theta_1,\theta_2)\in(-\pi,\pi]\times(-\pi,\pi]:~%
\frac{2\pi}{3}<|\theta_1+\theta_2|<\frac{4\pi}{3}\}\subset S_1(\mathbf{z}%
^\ast)\subset S_q(\mathbf{z}^\ast),\quad\forall q\in(0,1].
\end{equation*}
Fig. \ref{fig:ring:n} shows that as the number of neurons $n$ from the ring
increases, the stability domain $S_1(\mathbf{z}^\ast)$ approaches the set $%
D_{\theta_1,\theta_2}$ (light tan colored region). In the dark blue region
we have $q^\star<0.2$, meaning that the steady state $\mathbf{z}^\ast$ is
asymptotically stable only for small values of the fractional order $q$. It
can be observed that as the combination of parameters $(\theta_1,\theta_2)$
approaches the line $\theta_1+\theta_2=0$, the critical value $q^\star$
decreases towards $0$. In fact, if $\theta_1+\theta_2=0$ and the number of
neurons $n$ is sufficiently large, the steady state $\mathbf{z}^\ast$ is
unstable for any $q\in(0,1]$. Indeed, as
\begin{equation*}
\lambda_p=-1+2\cos\left(\theta_1+\frac{2p\pi}{n}\right)\in\mathbb{R}%
,\quad\forall~p=\overline{1,n}
\end{equation*}
when $n$ is sufficiently large, it can be easily shown that there exists $%
p\in\{0,1,...,n-1\}$ such that $\lambda_p>0$.

\begin{figure}[tbp]
\centering
\begin{tabular}{ccc}
\includegraphics[width=0.47\textwidth]{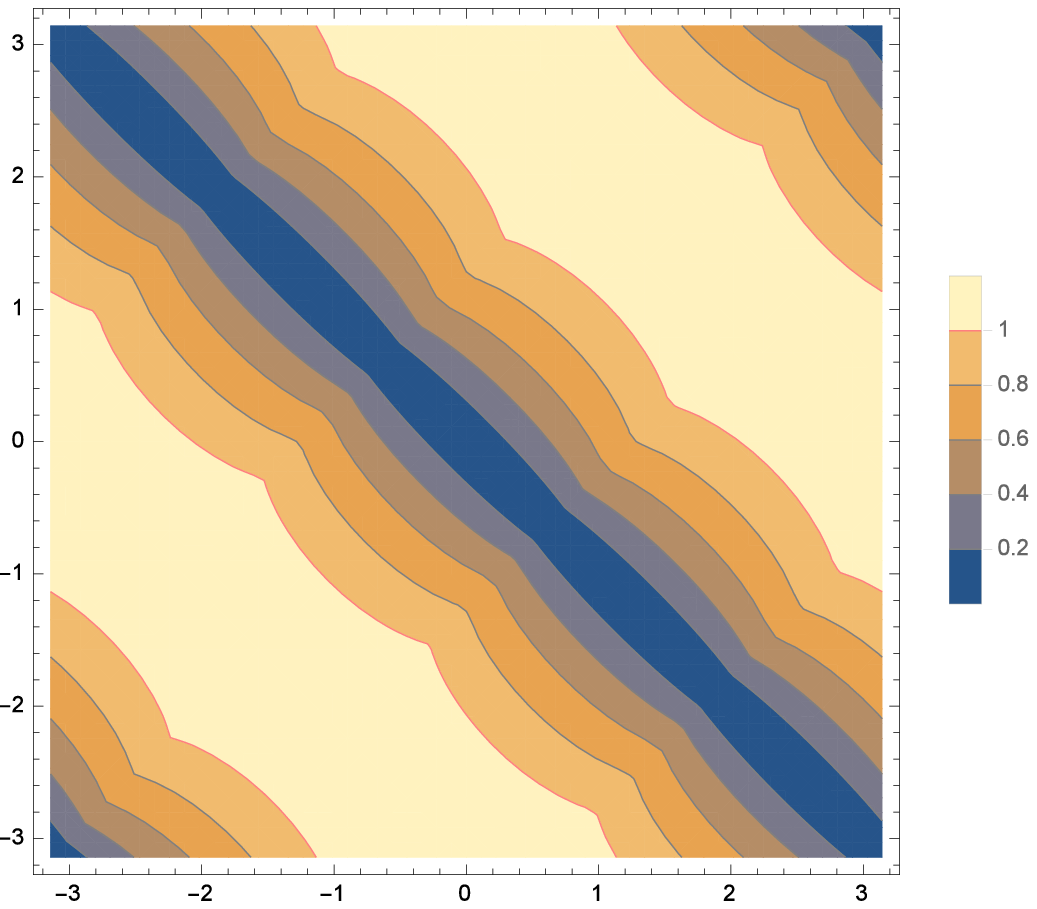} & %
\includegraphics[width=0.47\textwidth]{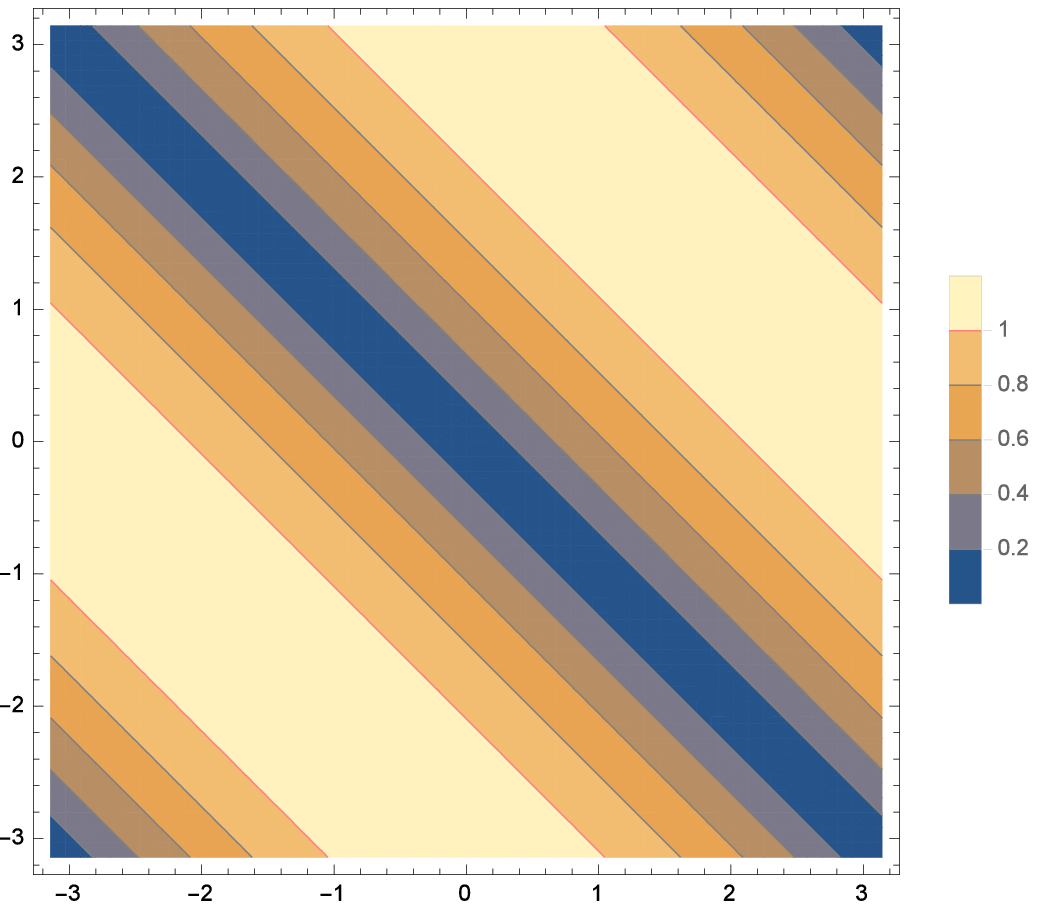} &  \\
$q^\star(\theta_1,\theta_2)$ when $n=5$ & $q^\star(\theta_1,\theta_2)$ when $%
n=100$ &
\end{tabular}%
\caption{Density plot of critical values $q^\star$ with respect to $(\protect%
\theta_1,\protect\theta_2)$, when $\protect\alpha=-1$, $\protect\beta=\exp(i%
\protect\theta_1)$, $\protect\gamma=\exp(i\protect\theta_2)$, for a small
number of neurons ($n=5$, left) and for a large number of neurons ($n=100$,
right), respectively. }
\label{fig:ring:n}
\end{figure}
\end{ex}

In the following, we exemplify a complex-valued neural network with an
infinity of steady states.

\begin{ex}
We consider the following complex-valued neural network of $n=3$ neurons
with ring structure:
\begin{equation}  \label{ex.ring}
\left\{%
\begin{array}{l}
^cD^qz_1(t)=-z_1(t)+T_0g(z_1(t))+T_1g(z_2(t))+T_2g(z_3(t)) \\
^cD^qz_2(t)=-z_2(t)+T_2g(z_{1}(t))+T_0g(z_2(t))+T_1g(z_{3}(t)) \\
^cD^qz_3(t)=-z_3(t)+T_1g(z_1(t))+T_2g(z_{2}(t))+T_0g(z_3(t))%
\end{array}
\right.
\end{equation}
In this neural network, $a=1$, $T_{11}=T_{22}=T_{33}=T_0=1-2i$, $%
T_{12}=T_{23}=T_{31}=T_1=1+i$ and $T_{13}=T_{21}=T_{32}=T_2=i$. The
activation function is $g(z)=\displaystyle\frac{z}{1+|z|}$.

The state $(z_1,z_2,z_3)^T$ is a steady state of system (\ref{ex.ring}) if
and only if it is a solution of the following nonlinear algebraic system
with complex coefficients:
\begin{equation}  \label{ex.ring.ss}
\left\{%
\begin{array}{l}
z_1=T_0g(z_1)+T_1g(z_2)+T_2g(z_3) \\
z_2=T_2g(z_1)+T_0g(z_2)+T_1g(z_3) \\
z_3=T_1g(z_1)+T_2g(z_2)+T_0g(z_3) \\
\end{array}
\right.
\end{equation}
From the first equation of (\ref{ex.ring.ss}) and $|g(z)|\leq 1$, it follows
that
\begin{equation*}
|z_1|\leq |T_0|+|T_1|+|T_2|=1+\sqrt{2}+\sqrt{5}=r.
\end{equation*}
We get similar results for $|z_2|$ and $|z_3|$, and hence, all the steady
states of system (\ref{ex.ring}) are inside the bounded set $B(0,r)\times
B(0,r)\times B(0,r)\subset\mathbb{C}^3$, where $B(0,r)$ denotes the ball
centered at the origin, of radius $r=1+\sqrt{2}+\sqrt{5}$ from the complex plane.

Particular steady states of system (\ref{ex.ring}) with equal components,
i.e. of the form $(z^\ast,z^\ast,z^\ast)^T$, can be found by solving the
equation
\begin{equation*}
-z+(T_0+T_1+T_2)g(z)=0
\end{equation*}
which is equivalent to $z=2g(z)$. It follows that $(z^\ast,z^\ast,z^\ast)^T$
is a steady state of system (\ref{ex.ring}) if and only if either $z^\ast=0$
(trivial equilibrium) or $|z^\ast|=1$ (i.e $z^\ast$ belongs to the unit
circle of the complex plane). We deduce that there are infinitely many
steady states for the system (\ref{ex.ring}).

In particular, we focus our attention on the stability of the trivial steady
state $\mathbf{z}^\ast=(0,0,0)^T$. As $g^{\prime }(0)=1$, based on (\ref%
{eq.alpha.ring}), (\ref{eq.beta.ring}) and (\ref{eq.gamma.ring}), we obtain $%
\alpha=-2i$, $\beta=T_1=1+i$ and $\gamma=T_2=i$. The first eigenvalue of the
Jacobian matrix is $\lambda_0=\alpha+\beta+\gamma=1$, and therefore, the
trivial solution is unstable for any fractional order $q\in(0,1)$. Indeed,
numerical computations (see Fig. \ref{fig:ring}) show that the trajectory of
the system with an initial condition from a small neighborhood of the
trivial steady state converges to one of the steady states of the form $%
(z^\ast,z^\ast,z^\ast)^T$ with $|z^\ast|=1$.

\begin{figure}[tbp]
\centering
\includegraphics[width=0.6\textwidth]{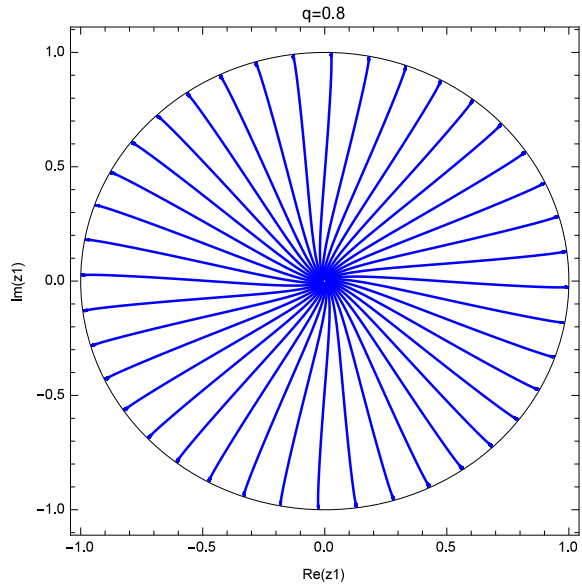}
\caption{Trajectories of system (\protect\ref{ex.ring}) with fractional
order $q=0.8$ and initial conditions from a small neighborhood of the null
solution.}
\label{fig:ring}
\end{figure}
\end{ex}

\section{Conclusions}

In the case of fractional-order complex-valued neural networks with ring or
hub structure, conditions for the stability and instability of a steady
state have been explored and various values of the fractional order $%
q\in(0,1)$ for which Hopf bifurcations may occur, have been identified.
Theoretical and numerical results presented in this paper show that
fractional-order complex-valued neural networks may exhibit rich dynamical
behavior in a neighborhood of a steady state, ranging from stability,
quasi-periodicity occurring due to a supercritical Hopf bifurcation, as well
as instability. Moreover, it has been pointed out that complex-valued neural
networks may possess an infinite number of steady states.

It is worth noting that chaotic behaviour has not been observed in our numerical simulations. This may be partly due to the activation function which has been used in the examples, as well as to the fact that in complex valued systems, a large number of steady states may be present (in our last example, we have an infinity of steady states). At least a part of these steady states may be asymptotically stable, such that the union of their domains of attraction include the whole phase-space. Assessment of chaotic behavior in fractional- or integer-order complex-valued neural networks with or without delays will be the topic of a future paper.

These simple neural network structures can be regarded as a test bed for
understanding the dynamics of more complicated structures. Applications of
such neural networks in pattern recognition and classification, intelligent
image processing, nonlinear filtering, brain-computer interfaces or time
series prediction may also constitute challenging directions for future research.

The extension of the results presented in this paper to more general fractional-order quaternion-valued neural networks \cite{Liu2016quaternion}, Clifford-valued neural networks \cite{Liu2016clifford} or complex-valued switched interval neural networks \cite{cao2013switched} is also an area worth exploring in future works.

\section*{Acknowledgement}

The authors are grateful to the editor and the referees for their helpful comments and suggestions which have improved the quality of the paper.

This work was supported by a grant of the Romanian National Authority for
Scientific Research and Innovation, CNCS-UEFISCDI, project no.
PN-II-RU-TE-2014-4-0270.

\appendix
\section{Proof of Lemma 1.}

We first assume that $A=\theta_1-\frac{\theta_2}{2}\notin\left\{-\pi,-\frac{\pi}{2},0,\frac{\pi}{2},\pi\right\}$.

Let us denote the roots of the polynomial $P(\lambda)$ by $\lambda_1 =r_1e^{i\phi_1}$ and $\lambda_2 =r_2e^{i\phi_2}$, with $r_1, r_2>0$ and $\phi_1,\phi_2\in(-\pi,\pi]$. We may assume, without loss of generality, that $\phi_1\geq\phi_2$. In this case, Vieta's formulas are:
$$
\left\{
  \begin{array}{ll}
    \lambda_1+\lambda_2=\alpha+\beta=\rho_1e^{i\theta_1} \\
    \lambda_1\cdot \lambda_2=(\alpha\beta-\gamma)=\rho_2e^{i\theta_2}
  \end{array}
\right.
$$
Taking the absolute value in the second formula, it follows that $r_1r_2=\rho_2$ and hence, the second Vieta's formula simplifies to
$$e^{i(\phi_1+\phi_2)}=e^{i\theta_2}$$ which leads to $\phi_1+\phi_2=\theta_2+2k\pi$, where $k\in\mathbb{Z}$. However, since $\phi_1+\phi_2\in(-2\pi,2\pi]$ and $\theta_2\in(-\pi,\pi]$, it follows that only three distinct cases have to be considered: $k=-1$ (if $\theta_2>0$), $k=0$ and $k=1$ (if $\theta_2<0$).

Denoting $\phi=\phi_1-\phi_2\in[0,2\pi)$, and dividing the first Vieta's formula by $(-1)^k e^{i\theta_2/2}$, we obtain:
\begin{equation}\label{dem.1}
r_1 e^{i\phi/2}+r_2e^{-i\phi/2}=(-1)^k\rho_1e^{iA}.
\end{equation}
Taking the real part, it follows that
$$(r_1+r_2)\cos(\phi/2)=(-1)^k\rho_1\cos(A)$$
and therefore, $\cos(\phi/2)$ and $(-1)^k\cos(A)$ must have the same sign, and hence, as $\phi\in[0,2\pi)$ and $\textrm{sign}(\sin(\phi))=\textrm{sign}(\cos(\phi/2))$, it follows that:
\begin{equation}\label{ineq0}
\textrm{sign}\left((-1)^k\sin(\phi)\right)=\textrm{sign}(\cos(A)).
\end{equation}
Multiplying eq. (\ref{dem.1}) by $e^{i\phi/2}$ and taking the imaginary part of the resulting equation, we obtain:
$$r_1\sin(\phi)=(-1)^k\rho_1\sin\left(A+\frac{\phi}{2}\right).$$
Similarly, multiplying eq. (\ref{dem.1}) by $e^{-i\phi/2}$ and taking the imaginary part of the resulting equation, we obtain:
$$-r_2\sin(\phi)=(-1)^k\rho_1\sin\left(A-\frac{\phi}{2}\right).$$
As $r_1,r_2,\rho_1>0$, it follows that the following condition must be satisfied:
\begin{equation}\label{ineq1}
\textrm{sign}\left((-1)^k\sin(\phi)\right)=\textrm{sign}\left(\sin\left(\frac{\phi}{2}\pm A\right)\right).
\end{equation}
Combining the last two equalities, we get:
$$-r_1r_2\sin^2(\phi)=\rho_1^2\sin\left(A+\frac{\phi}{2}\right)\sin\left(A-\frac{\phi}{2}\right),$$
or equivalently, since $r_1r_2=\rho_2$,
$$2\rho_2(\cos^2(\phi)-1)=\rho_1^2\left(\cos(\phi)-\cos(2A)\right).$$
Therefore, $\cos(\phi)$ is a root of the polynomial
$$p(x)=x^2-ux+u\cos(2A)-1,\qquad\textrm{where }u=\frac{\rho_1^2}{2\rho_2}.$$
The discriminant of the second degree polynomial $p(x)$ is
\begin{equation*}
\Delta =u^{2}-4u\cos (2A)+4\geq (u-2)^{2}\geq 0.
\end{equation*}%
We have
\begin{align*}
p(-1)&=u[1+\cos(2A)]>0\\
p(\cos(2A))&=\cos^2(2A)-1<0\\
p(1)&=u(\cos(2A)-1)<0
\end{align*}
and hence, the smallest root of the polynomial $q(x)$ belongs to the interval $(-1,\cos(2A))$, while the largest root is greater than $1$. Therefore, we obtain
$$
\cos(\phi)=\frac{u-\sqrt{u^{2}-4u\cos (2A)+4}}{2}\in (-1,\cos(2A)).
$$
and hence,
\begin{equation*}
\phi =\pm \arccos\left(\frac{u-\sqrt{u^{2}-4u\cos (2A)+4}}{2}\right) +2l\pi=\pm 2B+2l\pi,\quad l\in\mathbb{Z}.
\end{equation*}
As $\cos(\phi)<\cos(2A)$, we get $\cos(2B)<\cos(2A)$, and therefore,
$$\sin(A+B)\sin(B-A)>0,$$
which means that $\sin(B\pm A)$ have the same signs. Moreover, as $B\in\left(0,\frac{\pi}{2}\right)$, it follows that
\begin{equation}\label{ineq2}
\textrm{sign}(\sin(B\pm A))=\textrm{sign}\left(\sin B\cos A\pm \cos B\sin A\right)=\textrm{sign}(\cos A).
\end{equation}

Moreover, since $\phi\in[0,2\pi)$, it follows that we may have two cases:
\begin{description}
\item[Case 1.] $\phi=2B\in[0,\pi)$. In this case, it is easy to see that (\ref{ineq0}) and (\ref{ineq2}) imply (\ref{ineq1}).
Moreover:
\begin{itemize}
  \item If $k=-1$ (only if $\theta_2>0$), we obtain $\displaystyle\phi_{1,2}=\frac{\theta_2}{2}-\pi\pm B$. It is easy to see that $\phi_{1,2}\in(-\pi,\pi]$ if and only if $\theta_2>2B$. Based on (\ref{ineq0}), this subcase holds only if $\cos(A)<0$, leading to (b.3).
  \item If $k=0$, we obtain $\displaystyle\phi_{1,2}=\frac{\theta_2}{2}\pm B\in (-\pi,\pi]$. Based on (\ref{ineq0}), this subcase holds only if $\cos(A)>0$. Therefore, we obtain case (a).
  \item If $k=1$ (only if $\theta_2<0$), we obtain $\displaystyle\phi_{1,2}=\frac{\theta_2}{2}+\pi\pm B$. It is easy to see that $\phi_{1,2}\in(-\pi,\pi]$ if and only if $\theta_2\leq-2B$. Based on (\ref{ineq0}), this subcase holds only if $\cos(A)<0$, leading to (b.1).
\end{itemize}
\item[Case 2.] $\phi=2\pi-2B\in[\pi,2\pi)$. In this case, based on (\ref{ineq0}) and (\ref{ineq2}), we obtain:
\begin{align*}
\textrm{sign}\left(\sin\left(\frac{\phi}{2}\pm A\right)\right)&=\textrm{sign}\left(\sin\left(\pi-B\pm A\right)\right)\\
&=\textrm{sign}\left(\sin\left(B\pm A\right)\right)=\textrm{sign}\left(\cos (A)\right)\\
&=\textrm{sign}\left((-1)^k\sin(\phi)\right),
\end{align*}
and hence, condition (\ref{ineq1}) is satisfied. Moreover,
\begin{itemize}
\item If $k=-1$ (only if $\theta_2>0$), we obtain $\displaystyle\phi_{1,2}=\frac{\theta_2}{2}-\pi\pm (\pi-B)$. We observe that $\phi_2=\frac{\theta_2}{2}+B-2\pi<-\pi$, which is absurd.
\item If $k=0$, we obtain $\displaystyle\phi_{1,2}=\frac{\theta_2}{2}\pm (\pi-B)$. It is easy to see that $\phi_{1,2}\in(-\pi,\pi]$ if and only if $-2B<\theta_2\leq 2B$. Based on (\ref{ineq0}), this subcase holds only if $\cos(A)<0$, leading to (b.2).
\item If $k=1$ (only if $\theta_2<0$), we obtain $\displaystyle\phi_{1,2}=\frac{\theta_2}{2}+\pi\pm (\pi-B)$. We observe that $\phi_1=\frac{\theta_2}{2}-B+2\pi>\pi$, which is absurd.
\end{itemize}
\end{description}

If $A=\theta_1-\frac{\theta_2}{2}\in\left\{-\frac{\pi}{2},\frac{\pi}{2}\right\}$, i.e. $\cos(A)=0$, it follows from (\ref{dem.1}) that
$$(r_1+r_2)\cos(\phi/2)=0,$$
and hence, we obtain $\phi_1-\phi_2=\phi=\pi$. Taking into account that $\phi_1+\phi_2=\theta_2+2k\pi$, the only possible solution is $\displaystyle \phi_{1,2}=\frac{\theta_2\pm\pi}{2}$. This is a special case included in point (a) of the Lemma.

If $A=\theta_1-\frac{\theta_2}{2}\in\{-\pi,0,\pi\}$, i.e. $\cos(A)=(-1)^p$, $\sin(A)=0$, taking the real and imaginary parts of (\ref{dem.1}) we obtain
$$
\left\{
  \begin{array}{ll}
    (r_1+r_2)\cos(\phi/2)=(-1)^{k+p}\rho_1\\
    (r_1-r_2)\sin(\phi/2)=0
  \end{array}
\right.
$$
From here, either $r_1=r_2$ or $\phi=0$.

In the first case, $r_1=r_2$, taking into account that $r_1r_2=\rho_2$, we obtain $r_1=r_2=\sqrt{\rho_2}$ and
$$\cos(\phi/2)=(-1)^{k+p}\frac{\rho_1}{2\sqrt{\rho_2}}.$$
and hence,
$$\cos(\phi)=2\cos^2(\phi/2)-1=\frac{\rho_1^2}{2\rho_2}-1=u-1.$$
We observe that this can only hold if and only if $\rho_1^2\leq 4\rho_2$. Moreover, as $\phi\in[0,2\pi)$, we have $\phi=\arccos(u-1)=2B$ or $\phi=2\pi-\arccos(u-1)=2(\pi-B)$. From here, the proof follows cases 1 and 2 as above.

In the second case, $\phi=0$, the above system shows that $k+p$ is an even number.

If $A=0$, then $k=0$ and it simply follows $\phi_{1,2}=\frac{\theta_2}{2}$, which is a particular case of point (a) of the Lemma.

If $A\in\{-\pi,\pi\}$, then $k=\pm 1$ and hence $\phi_{1,2}=\frac{\theta_2}{2}+\pi$ if $\theta_2<0$ and $\phi_{1,2}=\frac{\theta_2}{2}-\pi$ if $\theta_2>0$, in agreement with points (b.1) and (b.3) of the Lemma.


\begin{thebibliography}{69}
\expandafter\ifx\csname natexlab\endcsname\relax\def\natexlab#1{#1}\fi
\expandafter\ifx\csname url\endcsname\relax
  \def\url#1{\texttt{#1}}\fi
\expandafter\ifx\csname urlprefix\endcsname\relax\def\urlprefix{URL }\fi

\bibitem[{Fra(2007)}]{Fractance}
NanoDotTek, TM, 2007. What is fractance and why is it useful? Tech. Rep. NDT24-11-2007,
  \url{p.//www.nanodottek.com}.

\bibitem[{Anastasio(1994)}]{Anastasio}
Anastasio, T., 1994. The fractional-order dynamics of brainstem
  vestibulo-oculomotor neurons. Biological Cybernetics 72~(1), 69--79.

\bibitem[{Arena et~al.(2000)Arena, Fortuna, and Porto}]{Arena}
Arena, P., Fortuna, L., Porto, D., 2000. Chaotic behavior in noninteger-order
  cellular neural networks. Physical Review E 61~(1), 776--781.

\bibitem[{Baldi and Atiya(1994)}]{ring-Baldi-Atiya}
Baldi, P., Atiya, A.~F., 1994. How delays affect neural dynamics and learning.
  IEEE Transactions on Neural Networks 5~(4), 612--621.

\bibitem[{Bao et~al.(2016)Bao, Park, and Cao}]{Bao_2016}
Bao, H., Park, J.~H., Cao, J., 2016. Synchronization of fractional-order
  complex-valued neural networks with time delay. Neural Networks 81, 16--28.

\bibitem[{Barab\'{a}si and Albert(1999)}]{Barabasi_1999}
Barab\'{a}si, A.-L., Albert, R., 1999. Emergence of scaling in random networks.
  Science 286~(5439), 509--512.

\bibitem[{Boroomand and Menhaj(2009)}]{Boroomand}
Boroomand, A., Menhaj, M., 2009. Fractional-order Hopfield neural networks.
  Lecture Notes in Computer Science 5506 LNCS~(PART 1), 883--890.

\bibitem[{Bungay and Campbell(2007)}]{ring-Bungay-Campbell}
Bungay, S.~D., Campbell, S.~A., 2007. Patterns of oscillation in a ring of
  identical cells with delayed coupling. International Journal of Bifurcation
  and Chaos 17~(9), 3109--3125.

\bibitem[{Campbell et~al.(1999)Campbell, Ruan, and
  Wei}]{ring-Campbell-Ruan-Wei}
Campbell, S.~A., Ruan, S., Wei, J., 1999. Qualitative analysis of a neural
  network model with multiple time delays. International Journal of Bifurcation
  and Chaos in Applied Sciences and Engineering 9~(8), 1585--1595.

\bibitem[{Cao et~al.(2013)Cao, Alofi, Al-Mazrooei, and Elaiw}]{cao2013switched}
Cao, J., Alofi, A., Al-Mazrooei, A., Elaiw, A., 2013. Synchronization of
  switched interval networks and applications to chaotic neural networks. In:
  Abstract and Applied Analysis. Vol. 2013. Hindawi Publishing Corporation.

\bibitem[{Chen et~al.(2014)Chen, Zeng, and Jiang}]{Chen_2014}
Chen, J., Zeng, Z., Jiang, P., 2014. Global Mittag-Leffler stability and
  synchronization of memristor-based fractional-order neural networks. Neural
  Networks 51, 1--8.

\bibitem[{Chen et~al.(2013)Chen, Chai, Wu, Ma, and Zhai}]{Chen_2013}
Chen, L., Chai, Y., Wu, R., Ma, T., Zhai, H., 2013. Dynamic analysis of a class
  of fractional-order neural networks with delay. Neurocomputing 111, 190--194.

\bibitem[{Cottone et~al.(2010)Cottone, Paola, and Santoro}]{Cottone}
Cottone, G., Paola, M.~D., Santoro, R., 2010. A novel exact representation of
  stationary colored gaussian processes (fractional differential approach).
  Journal of Physics A: Mathematical and Theoretical 43~(8), 085002.
\newline\urlprefix\url{http://stacks.iop.org/1751-8121/43/i=8/a=085002}

\bibitem[{Diethelm et~al.(2002)Diethelm, Ford, and Freed}]{Diethelm}
Diethelm, K., Ford, N., Freed, A., 2002. A predictor-corrector approach for the
  numerical solution of fractional differential equations. Nonlinear Dynamics
  29~(1-4), 3--22.

\bibitem[{El-Saka et~al.(2009)El-Saka, Ahmed, Shehata, and El-Sayed}]{ElSaka}
El-Saka, H., Ahmed, E., Shehata, M., El-Sayed, A., 2009. On stability,
  persistence, and Hopf bifurcation in fractional order dynamical systems.
  Nonlinear Dynamics 56~(1-2), 121--126.

\bibitem[{Elwakil(2010)}]{Elwakil}
Elwakil, A., 2010. Fractional-order circuits and systems: An emerging
  interdisciplinary research area. IEEE Circuits and Systems Magazine 10~(4),
  40--50.

\bibitem[{Engheia(1997)}]{Engheia}
Engheia, N., 1997. On the role of fractional calculus in electromagnetic
  theory. IEEE Antennas and Propagation Magazine 39~(4), 35--46.

\bibitem[{Georgiou and Koutsougeras(1992)}]{Georgiou_1992}
Georgiou, G.~M., Koutsougeras, C., 1992. Complex domain backpropagation. IEEE
  transactions on Circuits and systems II: analog and digital signal processing
  39~(5), 330--334.

\bibitem[{Gray(2005)}]{Circulant}
Gray, R.~M., 2005. Toeplitz and Circulant Matrices: A Review. Now Publishers
  Inc.

\bibitem[{Guo(2005)}]{ring-Guo}
Guo, S., 2005. Spatio-temporal patterns of nonlinear oscillations in an
  excitatory ring network with delay. Nonlinearity 18~(5), 2391--2407.

\bibitem[{Guo and Huang(2003)}]{ring-Guo-Huang-4}
Guo, S., Huang, L., 2003. Hopf bifurcating periodic orbits in a ring of neurons
  with delays. Physica D: Nonlinear Phenomena 183~(1-2), 19--44.

\bibitem[{Guo and Huang(2006)}]{ring-Guo-Huang-3}
Guo, S., Huang, L., 2006. Non-linear waves in a ring of neurons. IMA Journal of
  Applied Mathematics (Institute of Mathematics and Its Applications) 71~(4),
  496--518.

\bibitem[{Guo and Huang(2007{\natexlab{a}})}]{ring-Guo-Huang-1}
Guo, S., Huang, L., 2007{\natexlab{a}}. Stability of nonlinear waves in a ring
  of neurons with delays. Journal of Differential Equations 236~(2), 343--374.

\bibitem[{Guo and Huang(2007{\natexlab{b}})}]{ring-Guo-Huang-2}
Guo, S.~J., Huang, L.~H., 2007{\natexlab{b}}. Pattern formation and
  continuation in a trineuron ring with delays. Acta Mathematica Sinica,
  English Series 23~(5), 799--818.

\bibitem[{Henry and Wearne(2002)}]{Henry_Wearne}
Henry, B., Wearne, S., 2002. Existence of Turing instabilities in a two-species
  fractional reaction-diffusion system. SIAM Journal on Applied Mathematics 62,
  870--887.

\bibitem[{Heymans and Bauwens(1994)}]{Heymans_Bauwens}
Heymans, N., Bauwens, J.-C., 1994. Fractal rheological models and fractional
  differential equations for viscoelastic behavior. Rheologica Acta 33,
  210--219.

\bibitem[{Hirose(2006)}]{Hirose_2006}
Hirose, A., 2006. Complex-valued neural networks. Springer Science \& Business
  Media.

\bibitem[{Hirose(2009)}]{Hirose_2009}
Hirose, A., 2009. Complex-valued neural networks: the merits and their origins.
  In: 2009 International Joint Conference on Neural Networks. IEEE, pp.
  1237--1244.

\bibitem[{Hirsch(1989)}]{Hirsch}
Hirsch, M., 1989. Convergent activation dynamics in continuous-time networks.
  Neural Networks 2, 331--349.

\bibitem[{Huang et~al.(2017)Huang, Cao, Xiao, Alsaedi, and Hayat}]{Huang_2017}
Huang, C., Cao, J., Xiao, M., Alsaedi, A., Hayat, T., 2017. Bifurcations in a
  delayed fractional complex-valued neural network. Applied Mathematics and
  Computation 292, 210--227.

\bibitem[{Ichise et~al.(1971)Ichise, Nagayanagi, and
  Kojima}]{Ichise_Nagayanagi_Kojima}
Ichise, M., Nagayanagi, Y., Kojima, T., 1971. An analog simulation of
  non-integer order transfer functions for analysis of electrode processes.
  Journal of Electroanalytical Chemistry 33, 253--265.

\bibitem[{Kaslik(2009)}]{Eva_IJCNN2009}
Kaslik, E., 2009. Dynamics of a discrete-time bidirectional ring of neurons
  with delay. In: Proceedings of the International Joint Conference on Neural
  Networks, Atlanta, GA, USA. pp. 1539--1546.

\bibitem[{Kaslik and Balint(2009)}]{KB6}
Kaslik, E., Balint, S., 2009. Complex and chaotic dynamics in a
  discrete-time-delayed Hopfield neural network with ring architecture. Neural
  Networks 22~(10), 1411--1418.

\bibitem[{Kaslik and Sivasundaram(2011)}]{ES-IJCNN-2011}
Kaslik, E., Sivasundaram, S., 2011. Dynamics of fractional-order neural
  networks. In: Proceedings of the International Joint Conference on Neural
  Networks, San Jose, California, USA, July 31-August 5, 2011. IEEE Computer
  Society Press, pp. 611--618.

\bibitem[{Kaslik and Sivasundaram(2012{\natexlab{a}})}]{Eva-Siva-non-periodic}
Kaslik, E., Sivasundaram, S., 2012{\natexlab{a}}. Non-existence of periodic
  solutions in fractional-order dynamical systems and a remarkable difference
  between integer and fractional-order derivatives of periodic functions.
  Nonlinear Analysis: Real World Applications 13~(3), 1489--1497.

\bibitem[{Kaslik and Sivasundaram(2012{\natexlab{b}})}]{Eva-Siva-4}
Kaslik, E., Sivasundaram, S., 2012{\natexlab{b}}. Nonlinear dynamics and chaos
  in fractional-order neural networks. Neural Networks 32, 245--256.

\bibitem[{Kaslik and Sivasundaram(2014)}]{Kaslik_2014_ICNPAA}
Kaslik, E., Sivasundaram, S., 2014. Differences between fractional-and
  integer-order dynamics. In: 10th International Conference on Mathematical
  Problems in Engineering, Aerospace and Sciences: ICNPAA 2014. Vol. 1637. AIP
  Publishing, pp. 479--486.

\bibitem[{Kilbas et~al.(2006)Kilbas, Srivastava, and Trujillo}]{Kilbas}
Kilbas, A., Srivastava, H., Trujillo, J., 2006. Theory and Applications of
  Fractional Differential Equations. Elsevier.

\bibitem[{Kitajima and Kurths(2009)}]{Kitajima-Kurths}
Kitajima, H., Kurths, J., 2009. Bifurcation in neuronal networks with hub
  structure. Physica A: Statistical Mechanics and its Applications 388~(20),
  4499--4508.

\bibitem[{Kra and Simanca(2012)}]{Kra2012circulant}
Kra, I., Simanca, S.~R., 2012. On circulant matrices. Notices of the AMS
  59~(3), 368--377.

\bibitem[{Kuroe et~al.(2003)Kuroe, Yoshid, and Mori}]{Kuroe_2003}
Kuroe, Y., Yoshid, M., Mori, T., 2003. On activation functions for
  complex-valued neural networks —existence of energy functions—. In:
  Artificial Neural Networks and Neural Information Processing—ICANN/ICONIP
  2003. Springer, pp. 985--992.

\bibitem[{Lakshmikantham et~al.(2009)Lakshmikantham, Leela, and Devi}]{Lak}
Lakshmikantham, V., Leela, S., Devi, J.~V., 2009. Theory of fractional dynamic
  systems. Cambridge Scientific Publishers.

\bibitem[{Li and Ma(2013)}]{Li_Ma_2013}
Li, C., Ma, Y., 2013. Fractional dynamical system and its linearization
  theorem. Nonlinear Dynamics 71~(4), 621--633.

\bibitem[{Liu et~al.(2016{\natexlab{a}})Liu, Xu, Lu, and
  Liang}]{Liu2016clifford}
Liu, Y., Xu, P., Lu, J., Liang, J., 2016{\natexlab{a}}. Global stability of
  Clifford-valued recurrent neural networks with time delays. Nonlinear
  Dynamics 84~(2), 767--777.

\bibitem[{Liu et~al.(2016{\natexlab{b}})Liu, Zhang, Lu, and
  Cao}]{Liu2016quaternion}
Liu, Y., Zhang, D., Lu, J., Cao, J., 2016{\natexlab{b}}. Global $\mu$-stability
  criteria for quaternion-valued neural networks with unbounded time-varying
  delays. Information Sciences 360, 273--288.

\bibitem[{Lu and Guo(2008)}]{ring-Lu-Guo}
Lu, X., Guo, S., 2008. Complete classification and stability of equilibria in a
  delayed ring network. Electronic Journal of Differential Equations 2008,
  1--12.

\bibitem[{Lundstrom et~al.(2008)Lundstrom, Higgs, Spain, and
  Fairhall}]{Lundstrom}
Lundstrom, B., Higgs, M., Spain, W., Fairhall, A., 2008. Fractional
  differentiation by neocortical pyramidal neurons. Nature Neuroscience
  11~(11), 1335--1342.

\bibitem[{Mainardi(1996)}]{Mainardi_1996}
Mainardi, F., 1996. Fractional relaxation-oscillation and fractional phenomena.
  Chaos Solitons Fractals 7~(9), 1461--1477.

\bibitem[{Matignon(1996)}]{Matignon}
Matignon, D., 1996. Stability results for fractional differential equations
  with applications to control processing. In: Computational Engineering in
  Systems Applications. pp. 963--968.

\bibitem[{Matsuzaki and Nakagawa(2003)}]{Matsuzaki}
Matsuzaki, T., Nakagawa, M., 2003. A chaos neuron model with fractional
  differential equation. Journal of the Physical Society of Japan 72~(10),
  2678--2684.

\bibitem[{Metzler and Klafter(2000)}]{Metzler}
Metzler, R., Klafter, J., 2000. The random walk's guide to anomalous diffusion:
  a fractional dynamics approach. Physics Reports 339~(1), 1 -- 77.

\bibitem[{Milo et~al.(2002)Milo, Shen-Orr, Itzkovitz, Kashtan, Chklovskii, and
  Alon}]{Milo}
Milo, R., Shen-Orr, S., Itzkovitz, S., Kashtan, N., Chklovskii, D., Alon, U.,
  2002. Network motifs: simple building blocks of complex networks. Science
  298.

\bibitem[{Nakagawa and Sorimachi(1992)}]{Nakagawa}
Nakagawa, M., Sorimachi, K., 1992. Basic characteristics of a fractance device.
  IEICE Transactions on Fundamentals of Electronics, Communications and
  Computer Sciences E75-A~(12), 1814--1819.

\bibitem[{Petras(2006)}]{Petras}
Petras, I., 2006. A note on the fractional-order cellular neural networks. In:
  IEEE International Conference on Neural Networks. pp. 1021--1024.

\bibitem[{Podlubny(1999)}]{Podlubny}
Podlubny, I., 1999. Fractional differential equations. Academic Press.

\bibitem[{Rakkiyappan et~al.(2015{\natexlab{a}})Rakkiyappan, Cao, and
  Velmurugan}]{Rakkiyappan_2015_fcvnn}
Rakkiyappan, R., Cao, J., Velmurugan, G., 2015{\natexlab{a}}. Existence and
  uniform stability analysis of fractional-order complex-valued neural networks
  with time delays. Neural Networks and Learning Systems, IEEE Transactions on
  26~(1), 84--97.

\bibitem[{Rakkiyappan et~al.(2016)Rakkiyappan, Sivaranjani, Velmurugan, and
  Cao}]{Rakkiyappan_2016}
Rakkiyappan, R., Sivaranjani, R., Velmurugan, G., Cao, J., 2016. Analysis of
  global o (t- $\alpha$) stability and global asymptotical periodicity for a
  class of fractional-order complex-valued neural networks with time varying
  delays. Neural Networks.

\bibitem[{Rakkiyappan et~al.(2014)Rakkiyappan, Velmurugan, and
  Cao}]{Rakkiyappan_2014_finite}
Rakkiyappan, R., Velmurugan, G., Cao, J., 2014. Finite-time stability analysis
  of fractional-order complex-valued memristor-based neural networks with time
  delays. Nonlinear Dynamics 78~(4), 2823--2836.

\bibitem[{Rakkiyappan et~al.(2015{\natexlab{b}})Rakkiyappan, Velmurugan, and
  Cao}]{Rakkiyappan_2015_stability}
Rakkiyappan, R., Velmurugan, G., Cao, J., 2015{\natexlab{b}}. Stability
  analysis of fractional-order complex-valued neural networks with time delays.
  Chaos, Solitons \& Fractals 78, 297--316.

\bibitem[{Sabatier and Farges(2012)}]{Sabatier_2012}
Sabatier, J., Farges, C., 2012. On stability of commensurate fractional order
  systems. International Journal of Bifurcation and Chaos 22~(04), 1250084.

\bibitem[{Sugimoto(1991)}]{Sugimoto}
Sugimoto, N., 1991. Burgers equation with a fractional derivative; hereditary
  effects on nonlinear acoustic waves. Journal of Fluid Mechanics 225,
  631--653.

\bibitem[{Wang et~al.(2014{\natexlab{a}})Wang, Yu, and Wen}]{Wang_2014}
Wang, H., Yu, Y., Wen, G., 2014{\natexlab{a}}. Stability analysis of
  fractional-order Hopfield neural networks with time delays. Neural Networks
  55, 98--109.

\bibitem[{Wang et~al.(2014{\natexlab{b}})Wang, Yu, and Wen}]{Wang_2014_ring}
Wang, H., Yu, Y., Wen, G., 2014{\natexlab{b}}. Stability analysis of
  fractional-order Hopfield neural networks with time delays. Neural Networks
  55, 98--109.

\bibitem[{Wang et~al.(2015)Wang, Yu, Wen, and Zhang}]{Wang_2015_ring}
Wang, H., Yu, Y., Wen, G., Zhang, S., 2015. Stability analysis of
  fractional-order neural networks with time delay. Neural Processing Letters
  42~(2), 479--500.

\bibitem[{Wei and Jiang(2006)}]{ring-Wei-Jiang}
Wei, J., Jiang, W., 2006. Stability and bifurcation analysis in a neural
  network model with delays. Dynamics of Continuous, Discrete and Impulsive
  Systems Series A: Mathematical Analysis 13~(2), 177--192.

\bibitem[{Xiao et~al.(2015)Xiao, Zheng, Jiang, and Cao}]{Xiao_2015}
Xiao, M., Zheng, W.~X., Jiang, G., Cao, J., 2015. Undamped oscillations
  generated by Hopf bifurcations in fractional-order recurrent neural networks
  with Caputo derivative. Neural Networks and Learning Systems, IEEE
  Transactions on 26~(12), 3201--3214.

\bibitem[{Zhou et~al.(2009)Zhou, Hu, and Li}]{Zhou_2009}
Zhou, S., Hu, P., Li, H., 2009. Chaotic synchronization of a fractional neuron
  network system with time-varying delays. In: 2009 International Conference on
  Communications, Circuits and Systems, ICCCAS 2009. pp. 863--867.

\bibitem[{Zhou et~al.(2008)Zhou, Li, and Zhu}]{Zhou_2008_1}
Zhou, S., Li, H., Zhu, Z., 2008. Chaos control and synchronization in a
  fractional neuron network system. Chaos, Solitons and Fractals 36~(4),
  973--984.

\bibitem[{Zhu et~al.(2008)Zhu, Zhou, and Zhang}]{Zhu_2008}
Zhu, H., Zhou, S., Zhang, W., 2008. Chaos and synchronization of time-delayed
  fractional neuron network system. In: Proceedings of the 9th International
  Conference for Young Computer Scientists, ICYCS 2008. pp. 2937--2941.

\end{thebibliography}

\end{document}